\documentclass[12pt, a4paper]{article}

\usepackage[T1]{fontenc}
\usepackage[utf8]{inputenc}
\usepackage[top=2.5cm, bottom=2.5cm, left=2cm, right=2cm]{geometry}
\usepackage{microtype}
\usepackage{enumitem}
\usepackage{amsmath}
\usepackage{amssymb}
\usepackage{amsthm}
\usepackage{graphicx}
\usepackage{hyperref}
\usepackage{float}
\usepackage[dvipsnames]{xcolor}
\usepackage{tikz}
\usetikzlibrary{arrows.meta, positioning, fit, backgrounds}
\usetikzlibrary{decorations.pathreplacing}
\usepackage{graphics}
\usepackage[capitalize,noabbrev]{cleveref}
\usepackage{algorithm}
\usepackage{algpseudocode}
\usepackage{aliascnt}
\usepackage{orcidlink}

\usepackage{pgfplots}
\usepgfplotslibrary{groupplots}
\pgfplotsset{compat=1.16}

\usepackage[backend=bibtex,sorting=anyt]{biblatex}
\renewbibmacro{in:}{}
\bibliography{References}

\newtheorem{theorem}{Theorem}[section]
\newtheorem*{Theorem}{Theorem}

\newaliascnt{corollary}{theorem}
\newtheorem{corollary}[corollary]{Corollary}
\aliascntresetthe{corollary}
\newtheorem*{Corollary}{Corollary}

\newaliascnt{proposition}{theorem}
\newtheorem{proposition}[proposition]{Proposition}
\aliascntresetthe{proposition}

\newaliascnt{claim}{theorem}

\aliascntresetthe{claim}

\theoremstyle{definition}
\newaliascnt{definition}{theorem}
\newtheorem{definition}[definition]{Definition}
\aliascntresetthe{definition}

\theoremstyle{remark}
\newaliascnt{remark}{theorem}
\newtheorem{remark}[remark]{Remark}
\aliascntresetthe{remark}

\theoremstyle{plain}
\newaliascnt{lemma}{theorem}
\newtheorem{lemma}[lemma]{Lemma}
\aliascntresetthe{lemma}

\theoremstyle{definition}
\newaliascnt{example}{theorem}
\newtheorem{example}[example]{Example}
\aliascntresetthe{example}

\Crefname{theorem}{Theorem}{Theorems}

\Crefname{corollary}{Corollary}{Corollaries}

\Crefname{proposition}{Proposition}{Propositions}

\Crefname{claim}{Claim}{Claims}

\Crefname{definition}{Definition}{Definitions}

\Crefname{remark}{Remark}{Remarks}

\Crefname{lemma}{Lemma}{Lemmas}

\Crefname{example}{Example}{Examples}

\newcommand{\od}{\stackrel{\mbox {\tiny {def}}}{=}}

\title{LandscapeSHAP: Which Persistent Homology Class Gets the Credit?}
\author{Nikola Mili\'cevi\'c, \orcidlink{0000-0002-2896-7374}\\
Department of Mathematical Sciences\\
The University of Texas at El Paso, El Paso, TX 79968, USA\\
\texttt{nmilicevic@utep.edu}}
\date{}
\begin{document}
	\maketitle
	
	\begin{abstract}
		Shapley values, a solution concept from cooperative game theory, have recently become a standard tool for feature credit allocation in machine learning. They provide an axiomatically justified method to fairly distribute a model's prediction among the features of data. To the best of our knowledge, they have not yet been applied to explain machine learning models trained on features from topological data analysis. We develop what we believe is the first such approach, focusing on the persistence landscape featurization of persistence diagrams. Because each landscape coordinate is a rank statistic (the $k$-th largest tent-function value across all diagram points at a given filtration threshold), crediting a model's prediction back to individual persistent homology classes (persistence diagram points) is nontrivial. In particular, the naive coalition sampling estimators do not respect the rank-based structure because duplicate topological features would receive a winner-take-all credit instead of sharing. 
		
		We introduce LandscapeSHAP, a method for fair credit allocation to persistence diagram points based on a model's prediction. Our method represents each landscape coordinate as an integral over a parametrized family of cooperative games and solves each with a pivotal player type argument. For linear models on persistence landscapes, LandscapeSHAP has a closed form expression that gives the exact Shapley value of every persistence diagram point. In particular, there is no coalition sampling required. The four Shapley ``fairness" axioms (Efficiency, Symmetry, Linearity and Null Player) uniquely characterize this credit allocation for \emph{any} model, not only linear ones. For a general nonlinear model, this unique value can only be calculated exactly from its defining coalition averaging formula, which requires considering all $2^N$ many coalitions, where $N$ is the number of points in the persistence diagram. This is computationally intractable for persistence diagrams of realistic size. For now, we complement the exact linear model result with an efficient Monte Carlo sampling of permutations. We give convergence rates in terms of number of samples needed to approximate to a desired degree of accuracy. We also prove stability results for the LandscapeSHAP credit allocation, for any model. Finally, we showcase the usefulness and explainability that LandscapeSHAP brings to topological data analysis with a regression and classification task. 
	
	\end{abstract}
	
	\section{Introduction}
	
	Persistent homology is a central tool of topological data analysis (TDA). It is a tool that captures how the homology of a filtered space changes with respect to the filtration parameter. In applications this information is encoded in a persistence diagram. Persistence diagrams are a multiset of points $\{(b_i,d_i)\}_{i=1}^N$ in the halfspace $\mathbb{R}^2_{y\ge x}=\{(x,y)\,|\, y\ge x\}$, where the first coordinate represents the  birth of a homology class and the second coordinate represents the time of death. Persistence diagrams have several metrics associated with them which allows for clustering analysis and hypothesis testing with data \cite{Fasy2014}. Unfortunately persistence diagrams are not vectors (or elements of a Hilbert space) \cite{Mileyko2011} so any more serious tools from statistical analysis such as averages, expectations, correlations, and machine learning algorithms are not readily available. Additionally, there are no coarse (let alone isometric) embeddings of persistence diagrams into a Hilbert space \cite{Bubenik2020wagner}.
	
	Persistence landscapes are a featurization of persistence diagrams, based on the underlying rank function, introduced by Bubenik \cite{bubenik2015statistical}. They are particularly nice because they live in a separable Banach space and they satisfy a Strong Law of Large Numbers and a Central Limit Theorem \cite[Theorems 9 and 10]{bubenik2015statistical}. There are many other featurizations besides landscapes, e.g., persistence images \cite{JMLR:v18:16-337}, Betti curves \cite{Umeda2017}, tropical coordinates \cite{Kalinik2018}. The featurization of persistence diagrams allowed for a plethora of applications where TDA can now be combined with machine learning algorithms, such as classification and regression analysis on real world datasets \cite{DONUT}.
	
	Despite all of that, a large problem with machine learning methods is their inherent high dimensionality of learned parameters which obscures insight into how a model is making its decision \cite{Rudin2019,Lipton2018}. This is important for interpretability and hypothesis generation and testing. Persistent homology features are often high dimensional themselves. For example a persistence landscape with $20$ layers and $100$ grid point discretization is a vector in a $2000$-dimensional space. Therefore it is very hard to interpret how a trained model is making its predictions. Even if one understood how every coordinate of a persistence landscape affects a model's prediction, pulling back this information to the persistent homology classes themselves, i.e., the points in the persistence diagram, in some principled way is for the most part unexplored territory. 
	
	The Shapley value, a solution concept from cooperative game theory, is a method for fairly distributing the worth of a game among a group of players who have collaborated \cite{Shapley1951,Shapley1953}. Thinking of a model as a cooperative game, where the features of the data are the players, it is possible to use Shapley values for fair credit allocation to features in machine learning \cite{NIPS2017_8a20a862} (\Cref{fig:tda_pipeline}). 	
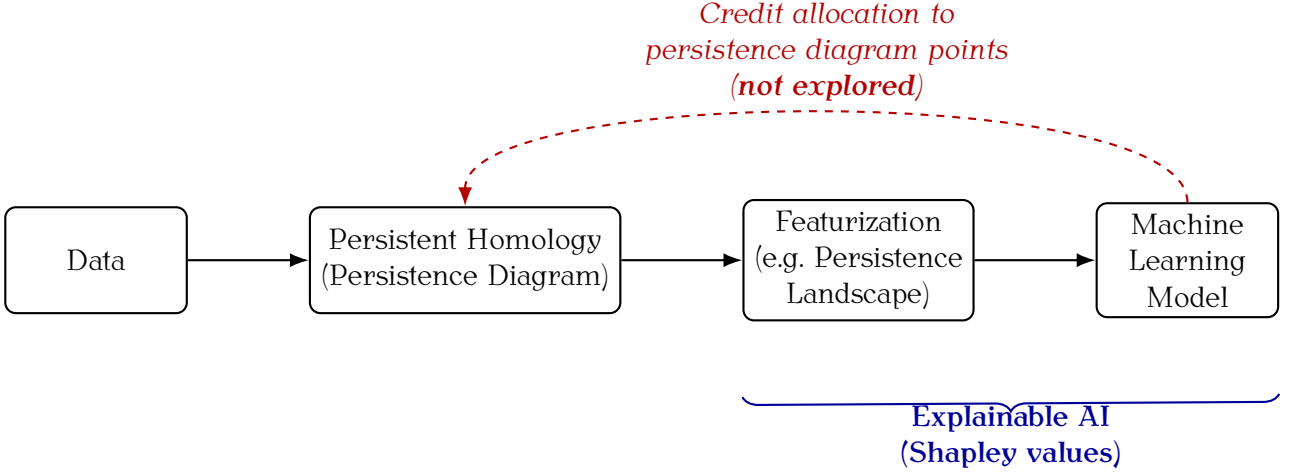
\begin{figure}[H]
	\centering
	\begin{tikzpicture}[
    	node distance = 1.6cm,
    	box/.style = {
        	rectangle, rounded corners, draw=black, thick,
        	minimum width=2.4cm, minimum height=1.4cm,
        	align=center, font=\small
    	},
    	arrow/.style = {-{Latex[length=2.5mm]}, thick},
    	dashedarrow/.style = {-{Latex[length=2.5mm]}, thick, dashed, red!70!black},
    	xaibrace/.style = {decorate, decoration={brace, amplitude=6pt, mirror}, thick, blue!60!black}
	]

	\node[box] (data) {Data};
	\node[box, right=of data] (ph) {Persistent Homology\\(Persistence Diagram)};
	\node[box, right=of ph] (feat) {Featurization\\(e.g.\ Persistence\\ Landscape)};
	\node[box, right=of feat] (ml) {Machine\\Learning\\Model};

	\draw[arrow] (data) -- (ph);
	\draw[arrow] (ph) -- (feat);
	\draw[arrow] (feat) -- (ml);
	
	\draw[dashedarrow]
    	(ml.north) .. controls +(0,1.6) and +(0,1.6) .. (ph.north)
    	node[midway, above, font=\small\itshape, red!70!black, align=center]
    	{Credit allocation to\\persistence diagram points\\ (\textbf{not explored})};

	\draw[xaibrace] ([yshift=-1.0cm]feat.south west) -- ([yshift=-1.0cm]ml.south east)
    	node[midway, yshift=-0.55cm, font=\small\bfseries, blue!60!black, align=center]
    	{Explainable AI\\(Shapley values)};

	\end{tikzpicture}
	\caption{The TDA to machine learning pipeline. Explainable AI methods based on Shapley
	values are well established for attributing model predictions back to
	featurization coordinates (blue brace), but pulling this credit back
	further to the individual points of the persistence diagram (red dashed
	arrow) has not been  addressed.}
	\label{fig:tda_pipeline}
\end{figure}	

In the last decade the machine learning community has embraced this approach and now Shapley values, or some extension based on them, are the most popular method in Explainable AI \cite{Li2024}. However, this methodology has not yet been integrated with the TDA pipeline (\Cref{fig:tda_pipeline}). We present the first such approach by introducing LandscapeSHAP, a method that fairly allocates credit to points in a persistence diagram by viewing persistence landscapes as a collection of order-statistic cooperative games. We believe this work will open new research directions in TDA that rely for the very first time on cooperative game theory (see \Cref{section:discussion}). We also hope for this manuscript to serve as an invitation to cooperative game theory for applied and computational topologists, researchers in explainable machine learning, and visualization experts.
 
 	\subsection*{Our contributions}
	
	The main contribution of this manuscript is the realization that the persistence diagram points can be viewed as players in a collection of cooperative games given by the persistence landscape. In particular, let $\mathcal{D}=\{(b_i,d_i)\}_{i=1}^N$ be a persistence diagram, and let $S\subset \{1,\dots, N\}$ determine a coalition of persistence diagram points from $\mathcal{D}$. Let $\lambda$ be a persistence landscape obtained from $\mathcal{D}$. Let $K$ be the number of layers of the landscape function we are using. Then the landscape is a function 
	\[\lambda: \{1,\dots, K\}\times \mathbb{R}\to \mathbb{R}.\] 
	 Let $\ell_i(m):\mathbb{R}\to \mathbb{R}$ be the tent function associated with persistence diagram point $(b_i,d_i)$ (which are used for constructing the landscape $\lambda$). In particular
	 \[\ell_i(m)\od\min(\max(0,m-b_i),\max(0,d_i-m)),\]
	 and the landscape value $\lambda(k,m)$ is the $k$-th largest value in the set $\{\ell_i(m)\}_{i=1}^N$. The novelty of our work is the observation that we have the cooperative game $v_{(k,m)}$, for each $k\in \{1,\dots ,K\}$ and $m\in \mathbb{R}$, which is defined for a coalition $S\subset \{1,\dots, N\}$ of persistence diagram points by 
	
	\[v_{(k,m)}(S)\od\begin{cases}
		\text{the k-th largest value among }\{\ell_i(m)\,|\, i\in S\}, & |S|\ge k\\
		0, & |S|< k
	\end{cases}.\]
	
	In other words, $v_{(k,m)}(S)$ is the value of the $k$-th landscape at the point $m$, when using only the points in $S$ to construct the persistence landscape (padding with $0$s wherever necessary). To compute the Shapley value of each player $i\in \{1,\dots ,N\}$ in this game, requires considering the marginal contribution of $i$ for the value of this game to every possible coalition $S$ (\Cref{def:shap_value}). Given $N$ points in a persistence diagram, from the definition, this requires summing over $2^{N-1}$ many coalitions, which is not feasible unless $N$ is small (e.g., $N\le 15$). However, realistic sizes of persistence diagrams are often in the hundreds of points $(N\ge 100)$. Our first major result is that no sampling is required to find the exact Shapley values for each player of the game $v_{(k,m)}$ and instead a closed form solution exists. 
	
	\begin{Theorem}[\Cref{theorem:closed_form_landscapes}]
			The Shapley value of the player $i$ in the game $v_{(k,m)}$ is 
			\[\varphi_i(v_{(k,m)})=\sum_{j=\max (r_i,k)}^N\dfrac{t_{(j,m)}-t_{(j+1,m)}}{j}.\]
	\end{Theorem}
	
	Here $t_{(1,m)}\ge \cdots \ge t_{(N,m)}$ (with $t_{(N+1,m)}=0$) is the descending ordering of the tent values $\{\ell_i(m)\}_{i=1}^N$ at the point $m$ and $r_i$ is the rank (index) of the player $i$ in this order. Furthermore, by Shapley uniqueness theorem (\Cref{theorem:shap_unique}), this closed form is the \emph{only way} to fairly allocate credit to each player, in the cooperative game $v_{(k,m)}$.
	
	For $k=1$, the game $v_{(1,m)}$ is the classical airport game of Littlechild and Thompson \cite{Littlechild1977}, and the formula above reduces to the Littlechild-Owen expression \cite{Littlechild1973}. To our knowledge, the closed form for general $k$ is new. As a byproduct, it gives a closed form for the Shapley values of top $k$ sum games \cite{Bhagat2014}, for which only a polynomial algorithm was previously known.
		
	\begin{Corollary}[\Cref{corollary:uniqueness_landscape_games}]
		For each ordered pair $(k,m)$, $\varphi(v_{(k,m)})$ is the only credit allocation rule satisfying Efficiency, Symmetry, Linearity and Null player axioms for the game $v_{(k,m)}$.
	\end{Corollary}

	A major aspect of persistent homology is the various stability theorems that guarantee if the data does not change by more than $\varepsilon$, the persistence diagrams will not change more than $\varepsilon$ in the bottleneck distance $d_B$ \cite{CohenSteiner2006,Bubenik2024,Chazal2016,Bauer2014}. A natural question to ask is, if the persistence diagram changes by some $\varepsilon$, by how much does the Shapley value of each persistence diagram point change? We show that the Shapley values of the cooperative games $v_{(k,m)}$ are $2$-Lipschitz with respect to the (potential) noise in persistence diagrams.
	
	Assume that persistence diagrams $D$ and $D'$ have an equal number of points $N$ with a fixed optimal correspondence between them that realizes their bottleneck distance $\varepsilon=d_B(D,D')$. Let $\varphi_r(v_{(k,m)}(D))$ and $\varphi_r(v_{(k,m)}(D'))$ be the Shapley values of the game $v_{(k,m)}$ on diagrams $D$ and $D'$ of players at rank $r$, respectively. Then we have the following.
					
	\begin{Theorem}[\Cref{theorem:stability_for_landscape_games}]
		For every rank $r$, we have the following inequality
		\[|\varphi_r(v_{(k,m)}(D))-\varphi_r(v_{(k,m)}(D'))|\le \left(\dfrac{2}{\max(r,k)}-\dfrac{1}{N}\right)\varepsilon\le 2\varepsilon,\]
		where $\varepsilon=d_B(D,D')$. 
	\end{Theorem}
		
	Note that nothing so far assumed the existence of a model $f$ that is trained on persistence landscapes. All the results so far were about understanding the landscape games $v_{(k,m)}$ themselves. To make a landscape an input to a model a discretization of the landscape functions is needed. In particular a finite set of  ``grid points" $M\subset \mathbb{R}$ is chosen and we consider the restriction of the landscape function 
	\[\lambda:\{1,\dots, K\}\times M\to \mathbb{R}.\]
	
	Once a model comes into play, i.e. a function $f:\mathbb{R}^{K\times |M|}\to \mathbb{R}$, it gives rise to a new cooperative game. Namely the game associated to the model $f$, is $v_f$ where the value of a coalition of persistence diagram points $S$ is given by
	\[v_f(S)\od f(\Lambda(S))-f(\Lambda(\varnothing)),\] 
	where $\Lambda$ is the transformation of persistence diagram into a persistence landscape with $K$ layers and grid points in $M$. 
	In other words, assume that $f$ is a model that has been trained on a collection of persistence landscapes, with its learned parameters now fixed. Evaluate the model on the new landscape obtained by restricting to the persistence diagram points in the coalition $S$. That is the value of the game $v_f$ for the coalition $S$. The Shapley values of player $i$ of this game, $\varphi_i(v_f)$, are once again defined by averaging the marginal contributions of player $i$ to $S$, over all coalitions (\Cref{def:shap_value}). For general nonlinear models no closed form solution exists as of yet, however it is conceivable one might derive closed form solutions or polynomial in time algorithms for certain families of nonlinear models in the future. This means that in order to compute the Shapley values of persistence diagram points, for nonlinear models, one has to evaluate $2^{N-1}$ many such coalitions, which is not feasible for realistic sizes of persistence diagrams. 
	
	Allocating credit to features for nonlinear models in machine learning is almost never exact, with the exception of tree ensembles \cite{Lundberg2020}. Most estimators of Shapley values use Monte Carlo sampling of permutations in order to approximate the Shapley values. Nevertheless, \Cref{theorem:shap_unique} guarantees that the Shapley values $\varphi_i(f)$ are the only fair way to allocate credit to persistence diagram points for the predictions of the model $f$, regardless of their computability.
	
	\begin{Corollary}[\Cref{corollary:fairness_axioms}]
		For any model $f$, there is exactly one credit allocation rule $\varphi(v_f)$ satisfying Efficiency, Symmetry, Linearity and Null player axioms for the game $v_f$.
	\end{Corollary}
	
	However, if $f$ is a linear model (e.g. (regularized) least squares, $1$-dim PCA projection, linear support vector machine), we prove that there is a closed form solution for calculating the Shapley values of the game $v_f$. 
	
	\begin{Theorem}[\Cref{theorem:closed_form_linear_model}]
	Let $f$ be a linear model trained on persistence landscapes, and let $w_{(k,m)}$ be the coefficients of the model for the landscape coordinate $(k,m)$. Then
		\[\varphi_i (v_f)=\sum_{k,m}w_{(k,m)}\varphi_i(v_{(k,m)}).\]
	\end{Theorem}
	
	In other words, once the Shapley values of the game $v_{(k,m)}$ are known, multiply them by the coefficient of the linear model $f$ at the landscape coordinate $(k,m)$ and then sum over all landscape coordinates to get the Shapley value of the game $v_f$ for the persistence diagram point $i$. For general (nonlinear) models we give an efficient Monte Carlo permutation sampling algorithm that estimates Shapley values of the game $v_f$ (\Cref{alg:nonlinear_model}). We also prove Hoeffding-type inequality estimates on the number of samples required to approximate the Shapley values of the game $v_f$ to a desired level of accuracy. In particular, let $R_i=||\ell_i||_1\cdot L_f$ be the product of the Lipschitz constant of $f$ and the $1$-norm of the $i$-th tent function. Then we have the following.
	 
	  \begin{Theorem}[\Cref{theorem:shapley_values_error_bound}]
	 	Fix a player $i$, and let $\hat{\varphi}_i(v_f)$ be the \Cref{alg:nonlinear_model} estimate after $n$ i.i.d. sampled permutations. For any $\varepsilon>0$, $\delta \in (0,1)$,
		\[n\ge \dfrac{2R_i^2}{\varepsilon^2}\log \dfrac{2}{\delta}\Longrightarrow \text{Pr}(|\hat{\varphi}_i(v_f)-\varphi_i(v_f)|\ge \varepsilon)\le \delta.\]
		 A guarantee holding simultaneously for all $N$ players, via a probability union bound, using a per player confidence level of $\frac{\delta}{N}$ only increases this
		 to $n= O(R^2_{\max}\varepsilon^{-2} \log(\dfrac{N}{\delta}))$ with $R_{\max} =\max_iR_i$ which is logarithmic and not polynomial, in $N$.
	 \end{Theorem}
	 
	 Finally, we also provide stability results for the Shapley values of the game $v_f$ for any model $f$. Assume once again that persistence diagrams $D$ and $D'$ have an equal number of points $N$ with a fixed optimal correspondence between them that realizes their bottleneck distance $\varepsilon=d_B(D,D')$. Let $\varphi_i(v_{f};D)$ and $\varphi_i(v_{f};D')$ be the Shapley values of the game $v_f$ on diagrams $D$ and $D'$ of players $i$, respectively.
	 
	\begin{Theorem}[\Cref{theorem:stability_nonlinear_model}]
		Let $f$ be $L_f$-Lipschitz with respect to the $1$-norm on the flattened $K\times |M|$ landscape vector. Then, for every player $i$ we have 
	\[|\varphi_i(v_f;D)-\varphi_i(v_f;D')|\le 2L_fK|M|\varepsilon.\]
	\end{Theorem}
	 
	We illustrate the power of our method in explaining exactly how a linear regressor trained to predict Gaussian curvature is making its decisions. We also test our method on a dynamical system classification task, and verify that the model is making reasonable decisions in its predictions based on real (observable) topological features in the data. We use a linear model as a potential prior for estimating how many permutations are needed for approximating Shapley values to a high degree of accuracy. This was used as an estimate for the number of samples needed to approximate Shapley values of a nonlinear model on the same data. We saw that the nonlinear model is making decisions based on a similar set of topological features as the linear model. Therefore, our method is a powerful visualization and explainability toolkit that can be incorporated in standard TDA pipelines and offers additional tests for a model's validity.  We also discuss how one might define cooperative games based on some other common featurizations of persistence diagrams in the literature.
	
	This manuscript is structured as follows. In \Cref{section:prelminaries} we recall the necessary background on Shapley values and persistent homology. In \Cref{section:pers_landscapes_games} we show how to compute the Shapley values of landscape games. In \Cref{section:linear_models} we show how to allocate credit via Shapley values to persistence diagram points based on predictions of linear models trained on persistence landscapes. In \Cref{section:nonlinear_models} we show how to approximate Shapley values for nonlinear models on persistence landscapes. In \Cref{section:experiments} we apply LandscapeSHAP to regression and classification problems empirically confirming its usefulness. In \Cref{section:other_featurizations} we discuss how to view persistence images and other additive featurizations as cooperative games. In \Cref{section:discussion} we discuss the implications of this work and possible future directions.

	\subsection*{Related works}
	The present work sits at the intersection of several active research programs, and it is worth orienting the reader with respect to each of them.
	
	\paragraph{Shapley values for explainable AI.}
	Shapley values provide a principled way to explain the predictions of machine learning models. By interpreting a model trained on a set of features as a value function on a coalition of players, Shapley values provide an axiomatically fair way to allocate credit to features for a prediction \cite{NIPS2017_8a20a862}. They can also be used to allocate credit to the uncertainty of a prediction \cite{watson2023predictive}. Furthermore,  Shapley values unify several other attributive methods such as locally interpretable model-agnostic explanations (LIME) \cite{Ribeiro2016}, DeepLIFT \cite{pmlr-v70-shrikumar17a} and layer-wise relevance propagation \cite{Bach2015}. This unified perspective is a consequence of their uniqueness (\Cref{theorem:shap_unique}). Additionally, Shapley values can be used in all three stages of machine learning modeling \cite[Section 5]{Li2024}:
	\begin{enumerate}[left=0pt]
		\item \textbf{Pre-modeling.} For feature selection and dimensionality reduction.
		\item \textbf{Mid-modeling.} For credit allocation in cooperative multi-agent reinforcement learning (MARL).
		\item \textbf{Post-modeling.} For data valuation and explaining model predictions.
	\end{enumerate}
	We hope that LandscapeSHAP becomes a tool that can be used in a similar fashion for machine learning with TDA.
	
	We (re)emphasize that all of this machinery is readily available and applicable to featurizations of persistence diagrams (\Cref{fig:tda_pipeline}). However, as practitioners of TDA we would like to explain model predictions in terms of the persistent homology classes, i.e. the persistence diagram points themselves. For example, imagine discovering that the persistence landscape coordinate $(5,23)$ has the largest Shapley value. This does not help us in any way determine the relevant persistent homology class for the model. Persistence diagrams are not vectors or elements of a Hilbert space and therefore the above described methods cannot be applied to them. Our manuscript is the first approach for solving this problem.
	
	\paragraph{Explainable TDA in biology.}	
	There has been a growing body of work on interpreting the outputs of TDA pipelines for biological data (see, e.g., \cite{math9151723,EDWARDS2021100367,Hartsock2025,Mishra2025}). The interpretation step in these approaches, however, is typically ad hoc, lacking the axiomatic uniqueness or stability guarantees a Shapley value based credit allocation provides. LandscapeSHAP, by contrast, does not depend on the application domain, and it comes with an axiomatically justified attribution rule, together with the stability and uniqueness results highlighted above.
	
	\paragraph{Explainable TDA via heatmaps on point clouds.}
	Bubenik, Wagner, and Yadav \cite{bubenik2026explainabletopologicaldataanalysis} tackle a different aspect of the explainability problem in TDA.  In particular, the goal is to allocate credit to the raw data itself (the pixels/voxels) from an endpoint classifier model. Their method works by replacing the integer coefficient representative cycles produced by a persistence algorithm with an averaged, real coefficient ``persistence heatmap". These heatmaps are Lipschitz stable and uniformly continuous. Their credit allocation does satisfy the efficiency axiom, where the players are the pixels of the point cloud image, but the other fairness axioms have not been examined. In contrast, our manuscript gives an axiomatic framework for fair credit allocation among persistence diagram points from an endpoint model. Together with our manuscript, this creates an opportunity of using the Shapley value credit allocation to persistence diagram points as the importance signal driving the persistence heatmap construction. This would potentially yield a stable, fair, axiomatically justified attribution that runs all the way from the model's prediction down to the original geometric data.
	
	\paragraph{Closed form Shapley values for structured games.} For games with special structure Shapley values can often be computed exactly and sometimes in closed form, e.g. the airport game, where the value of a coalition is determined by its largest member, i.e. the $1$-st order statistic, \cite{Littlechild1973,Littlechild1977}. The landscape games $v_{(k,m)}$ replace this maximum with the $k$-th order statistic and \Cref{theorem:closed_form_landscapes} extends the Littlechild-Owen formula to every $k$. \cite{Bhagat2014} studied the game whose value is the sum of the top $k$ values in a coalition (Data Mining Problem) and gave a polynomial time dynamic programming algorithm for its Shapley values. Since the Data Mining Problem game is a sum of our games, the linearity of Shapley values and \Cref{theorem:closed_form_landscapes} give the Shapley values for the Data Mining Problem in closed form.
	
	\paragraph{Deep learning on persistence diagrams.}
	In some cases, it is possible to use persistence diagrams directly as input to a deep learning architecture \cite{Carrire2019PersLayAN,reinauer2022persformer}, without any additional featurization. With these frameworks, computing the gradient of the network's output (or a chosen prediction score) with respect to each input point can produce a saliency map directly on the persistence diagram itself \cite{simonyan2014deep}. In \cite{Qin2023} the authors use persistence images, a featurization of persistence diagrams, with uniform weight $1$ at every pixel unlike standard persistence weights. They train a metric deep learning model whose goal is to learn the optimal pixel weights for the persistence images in order to get maximal separation between data classes. The metric deep learning model consists of a Convolutional Neural Network (CNN) with an attention module. The authors then apply the Grad-CAM explainability method in deep learning \cite{grad-cam} to produce a heatmap of importance on the persistence diagram. Unlike Shapley values in machine learning that allocate credit based on a model's prediction, Grad-CAM detects which pixels or regions in an image were important by looking at the gradients of the final layer of the neural network. Additionally, it might also be possible to combine the DeepSHAP method \cite{NIPS2017_8a20a862} with these approaches.
	
	\paragraph{Shapley values in applied topology.}
	There is recent work connecting the theory of Shapley values with applied topology, although in a very different direction than what was done in this manuscript.
	In \cite{Zhang2020}, the authors defined Shapley homology. More specifically, if $X$ is a finite metric space, let $\text{\v{C}}_{r}(X)$ be the \v{C}ech complex constructed from $X$ at scale $r\ge 0$. This definition depends on the scale $r$ chosen to construct the \v{C}ech complex and the dimension of homology of interest. In particular, this construction does not use persistent homology since the scale $r$ is assumed fixed. It assigns credit to an element of a finite metric space based on how much it contributes to the Betti number on average, across coalitions. It does not allocate credit to persistent homology classes nor does it aim to explain predictions of models that use persistent homology for learning.
	
	\section{Preliminaries}
	\label{section:prelminaries}
	Here we recall the necessary background on Shapley values and persistent homology.
	
	\subsection{Shapley values}
	
		Here we recall cooperative games and the Shapley values associated to players in a cooperative game. 
		
		\begin{definition}
			A \emph{cooperative game in characteristic function form} is an ordered pair $(N,v)$ where $N$ is a finite set, the set of players, and the characteristic function $v$, $v:2^N\to \mathbb{R}$ with $v(\varnothing)=0$.
		\end{definition}
	
		A subset $S$ of $N$ is called a \emph{coalition} (of players). The number $v(S)$ gives the worth of the coalition $S$ in the game. If there is no confusion about the set $N$, we will typically write $v$ for $(N,v)$.
		
		The Shapley value is one way to divide up the value created by a coalition between its members. It is a way to allocate ``credit" or worth to individual players with respect to how much they contribute to the value $v(S)$, for all possible coalitions $S$.
		
		\begin{definition}
		\label{def:shap_value}
			The \emph{Shapley value}, $\varphi_i(v)$, of player $i$ in a cooperative game $(N,v)$ is given by 
				\[\varphi_i(v)\od \dfrac{1}{|N|}\sum_{S\subset N\setminus \{i\}}\binom{|N|-1}{|S|}^{-1}(v(S\cup\{i\})-v(S)),\]
				where $|S|$ and $|N|$ are the number of players in $S$ and $N$, respectively, and the sum is over all subsets $S$ of $N$ not containing $i$, including the empty set.  
		\end{definition}
		
		The formula for Shapley values is quite intuitive. Indeed, given a coalition $S$, player $i$ should demand their contribution $v(S\cup\{i\})-v(S)$ as fair compensation. Taking the average of this contribution over all possible ways in which the coalitions can be formed gives the Shapley value of player $i$. An equivalent, permutation based definition that computes the same value is also commonly used.
		
		\begin{definition}
		\label{def:shap_value_permutation}
			Let $\pi$ be a permutation drawn uniformly at random from the set $S_N$ of all $|N|!$ orderings of $N$ and let $P_i^{\pi}=\{j\in N\,|\, \pi(j)<\pi(i)\}$ be the set of all players in $N$ that precede $i$ in $\pi$. Then the Shapley value of player $i$ is the permutation average:
			\[\varphi_i(v)\od \dfrac{1}{|N|!}\sum_{\pi \in S_N} (v(P_i^{\pi}\cup \{i\})-v(P_i^{\pi})).\]
		\end{definition}
		
		The calculation of Shapley values involves averaging marginal contributions over $2^{N-1}$ many coalitions which is computationally intractable unless for very small sets of players $N$. Nevertheless, the Shapley value has been widely studied from a theoretical point of view due to its many useful properties. In particular, the Shapley value is in fact the \emph{unique} ``fair" distribution of a player's worth in a game. By fair we mean an allocation of credit that satisfies the following fairness axioms.
		
		\begin{definition}
			Let $\psi_i(v)$ be a distribution of credit to player $i$ in a cooperative game $(N,v)$. We say that $\psi$ is \emph{fair} if it satisfies each of the following four axioms. 
			\begin{itemize}[left=0pt]
				\item \textbf{Efficiency.} $\sum_{i\in N}\psi_i(v)=v(N)$.
				\item \textbf{Symmetry.} If for all $S\subset N$ with $i,j\not\in S$ we have $v(S\cup\{i\})=v(S\cup\{j\})$, then $\psi_i(v)=\psi_j(v)$.
				\item \textbf{Linearity.} If two cooperative games with players in $N$ with characteristic functions $v$ and $w$ are combined into the cooperative game $v+w$, where $(v+w)(S)=v(S)+w(S)$, then $\psi_i(v+w)=\psi_i(v)+\psi_i(w)$.
				\item \textbf{Null player.} A player $i$ is \emph{null} if $v(S\cup\{i\})=v(S)$ for all coalitions $S$. If a player $i$ is null, then $\psi_i(v)=0$.
			\end{itemize}

		\end{definition}
		
		\begin{theorem}{\cite{Shapley1953,Shapley1951}}
		\label{theorem:shap_unique}
			The Shapley value is the only fair allocation of credit to players in a cooperative game.
		\end{theorem}
	
	\subsection{Persistent homology}
	
	Let $X= \{x_1,\dots ,x_n\}$ be a finite set of points equipped with a metric. A \emph{simplicial complex} $K(X)$ on $X$ is a collection of subsets of $X$, called \emph{simplices}, such that if $\sigma$ is a simplex and $\tau \subset \sigma$, then $\tau$ is a simplex as well. A \emph{filtration} is a family of simplicial complexes $\{K_r\}_{r\ge 0}$, indexed by a scale parameter $r$, such that $r\le r'$ implies $K_r \subset K_{r'}$. Two standard constructions build such a filtration from the finite metric space $X$, and both will be used in \Cref{section:experiments}.
	
	\paragraph{Vietoris–Rips filtration.} The Vietoris-Rips complex was introduced by Vietoris \cite{vietoris1927hoheren} as a means of associating a simplicial complex to a graph or a metric space. Given a finite metric space $(X,d)$ with no ambient coordinates required, the \emph{Vietoris-Rips complex} at scale $r$ is the simplicial complex defined by 
	\[\text{VR}_r(X) \od \{\sigma\subset X\, |\,  d(x,y) \le r\text{ for all }x,y\in \sigma\}.\]  
	
	 Due to its simple definition from pairwise distances and strong theoretical guarantees, the Vietoris–Rips complex has become a central construction in topological data analysis \cite{Carlsson2009}, with efficient computation made possible by algorithms such as Ripser \cite{Bauer2021Ripser}.

	\paragraph{Alpha complex filtration.} When $X\subset \mathbb{R}^d$ has genuine ambient coordinates, the \emph{alpha complex} at scale $r$ is the subcomplex of the Delaunay triangulation of $X$ consisting of those simplices $\sigma$ whose points share a common empty circumscribing ball of radius at most $r$. Because it sits inside the Delaunay triangulation, the alpha complex at any scale has only $O(n)$ simplices in the plane (as opposed to up to $O(2^n)$ for Vietoris–Rips), making it far cheaper computationally. However the alpha complex requires Euclidean coordinates rather than an arbitrary distance matrix, and its construction cost grows quickly with the ambient dimension $d$ \cite{1056714}.

	\paragraph{Persistent homology and persistence diagrams.} Applying the simplicial homology functor $H_p(\cdot; \mathbf{k})$ in a fixed degree $p$ to every complex in a given filtration, and to every inclusion $K_r\hookrightarrow K_{r'}$ , produces a sequence of $\mathbf{k}$ vector spaces and $\mathbf{k}$-linear maps induced by inclusions. Such a collection is called a \emph{persistence module}. Under mild finiteness conditions this module decomposes uniquely into a multiset of intervals $[b_i,d_i)$ \cite{Zomorodian2004}. This data represents a topological feature born at scale $b_i$, the first $r$ at which it appears in $H_p(K_r)$, and dies at scale $d_i>b_i$, the first $r$ at which it becomes homologous to an older feature or becomes a boundary. The \emph{persistence diagram} $D= \{(b_i,d_i)\}_{i=1}^N$ is the resulting multiset of birth-death pairs. 
	
	\paragraph{Persistence landscapes.} A persistence diagram is a multiset of points, not a vector, so it cannot be fed directly into a standard statistical or machine-learning model. The persistence landscape is a featurization method that sends a persistence diagram into a function which is an element of a separable Banach space introduced by Bubenik in \cite{bubenik2015statistical}. Consider a persistence diagram consisting of $N$ birth-death pairs $\{(b_i,d_i)\}_{i=1}^N$ that has been converted into a persistence landscape $\lambda$  with $K$ many layers. Each point $i$ in the persistence diagram contributes a tent function
	
		\[\ell_i(m)\od \min \left(\max\left(0,m-b_i\right),\max\left(0,d_i-m\right)\right)\ge 0,\]
		and the landscape value $\lambda(k,m)$ is the $k$-th largest value among $\{\ell_i(m)\}_{i=1}^N$, and is defined to be $0$ if fewer than $k$ points are present.  Alternatively, we can think of the landscape as a sequence of functions $\lambda_k:\mathbb{R}\to \mathbb{R}$, $k=1,\dots, K$.

	\section{Persistence landscapes as cooperative games}
	\label{section:pers_landscapes_games}
	Here we show how to view the persistence landscape function as a collection of cooperative games of coalitions of persistence diagram points. We show how to compute the Shapley values for these games with a closed form solution. Additionally, we show these credit allocations are stable.
	
	\subsection{Coalitions and games from tent functions}
		
		The  landscape $\lambda$ will eventually be an input into a (linear) model, which will give a particular prediction. We want to attribute this prediction back to the original diagram points $i$, not merely to the landscape coordinate $(k,m)$. The ``natural" way to define a coalition game over diagram points is the following. Given a subset $S\subset \{1,\dots, N\}$ of persistence diagram points, recompute the landscape as if only the points in $S$ existed. We argue that this is the most ``honest" notion of a point's contribution since removing a point causes the remaining points to shift rank and fill in the vacated layer indices with their own tent value and not with a placeholder. 
	
		\begin{remark}
		\label{remark:average_baseline}
		A cheaper alternative is to fix once and for all what player ``owns" which coordinates $(k,m)$ in the full landscape and simply substitute a baseline value when an owner is removed. A standard choice in machine learning is to replace missing features by the mean of the data; in our case the mean landscape value at $(k,m)$ (where the mean is over all the landscapes of all the persistence diagrams). This avoids recomputing landscapes for different coalitions. However it has two additional problems:	
		\begin{enumerate}[left=0pt]
			\item It freezes the rank/ownership with respect to the full coalition $S=\{1,\dots, N\}$ and ignores the re-ranking a true removal induces.
			\item The persistence diagram points with nearly identical tent values receive a winner-take-all credit assignment determined by an arbitrary tie-break, rather than sharing credit proportional to how interchangeable they are.
		\end{enumerate}
		\end{remark}
		
		Consider a landscape function $\lambda:\{1,\dots, K\}\times \mathbb{R}\to \mathbb{R}$. Let $(k,m)$ be a landscape coordinate and define the characteristic function
	
		\[v_{(k,m)}(S)\od \begin{cases}
		\text{the k-th largest value among }\{\ell_i(m)\,|\, i\in S\}, & |S|\ge k\\
		0, & |S|< k
		\end{cases}.\]
		In other words, $v_{(k,m)}(S)$ is exactly the $k$-th landscape value at the point $m$, evaluated on the subdiagram consisting of the points in $S$. For $S=\{1,\dots, N\}$, $v_{(k,m)}(S)$ is precisely the original landscape $\lambda$ at the coordinate $(k,m)$. For $S=\varnothing$, $v_{(k,m)}(S)=0$. We view $v_{(k,m)}(S)$ as the value of a cooperative game with coalition $S$ and we want to compute the Shapley value $\varphi_i(v_{(k,m)})$ of every player $i$ in this game. 
	
		\begin{remark}
			Note that $v_{(k,m)}$ is a game over the landscape coordinate $(k,m)$, with no dependency on any downstream model. Therefore, $\varphi_i(v_{(k,m)})$ is model-independent and it answers ``how is this one value fairly built up from the diagram points" and not ``how much does a particular model's prediction change because of point $i$". The model will enter into the picture later.
		\end{remark}
	
	\subsection{Closed form solution}
	
		Let $m\in \mathbb{R}$ and let $t_{(1,m)}\ge t_{(2,m)}\ge \cdots\ge t_{(N,m)}$ be the sorting of the tent values $\ell_i(m)$, for $i=1,\dots, N$, in decreasing order and set $t_{(N+1,m)}=0$. Let $r_i\in \{1,\dots, N\}$ denote the rank of player $i$ in this order, with $1$ being the largest. That is $r_i=j$ if there is an index $j$ such that $\ell_i(m)=t_{(j,m)}$. Then we have the following result. 	
		
		\begin{theorem}
		\label{theorem:closed_form_landscapes}
			The Shapley value of the player $i$ in the game $v_{(k,m)}$ is 
		
			\[\varphi_i(v_{(k,m)})=\sum_{j=\max (r_i,k)}^N\dfrac{t_{(j,m)}-t_{(j+1,m)}}{j}.\]
		\end{theorem}
	
		\begin{proof}
		For each $\tau\ge 0$, define the game $w_{\tau}$ on the player set $\{1,\dots, N\}$ by 
		\[w_{\tau}(S)\od \begin{cases}
			1, & |\{i\in S\,|\, \ell_i(m)>\tau\}|\ge k\\
			0, &  |\{i\in S\,|\, \ell_i(m)>\tau\}|< k
		\end{cases},\]
		for any coalition $S$. Let $S$ be fixed and let $M_S$ denote the $k$-th largest value among $\{\ell_i(m)\,|\, i\in S\}$, with $M_S=0$ if $|S|< k$. By construction, $w_{\tau}(S)=1$ exactly when $\tau< M_S$. Otherwise, $w_{\tau}(S)=0$. In particular, for a fixed $S$, $\tau\mapsto w_{\tau}(S)$ is precisely the indicator function of the interval $[0,M_S)$ and therefore
		\[\int_{0}^{\infty}w_{\tau}(S)d\tau=\int_{0}^{M_S} 1d\tau = M_S=v_{(k,m)}(S).\]
		This holds for every $S$ and hence $v_{(k,m)}=\int_{0}^{\infty}w_{\tau}d\tau$. 
		
		Let $\tau\ge 0$ be fixed and let $B(\tau)\od\left|\{i\in \{1,\dots, N\}\,|\, \ell_i(m)>\tau\}\right|$. In other words, $B(\tau)$ is the number of ``big" players with tent values above the threshold $\tau$, i.e., players $i$ with $\ell_i(m)>\tau$. For big players $i$ and $j$ and a coalition $S$ with $i,j\not\in S$, adding $i$ or $j$ increases the big player number by exactly $1$ and thus $w_{\tau}(S\cup\{i\})=w_{\tau}(S\cup\{j\})$. Therefore, big players are symmetric. Recall that symmetric players get equal Shapley values. For a ``small" player $i$ with $\ell_i(m)\le \tau$, this player does not contribute to $w_{\tau}(S)$ and is thus null. Recall that null players have Shapley value of $0$. 
		
		If $B(\tau)<k$, then $w_{\tau}=0$ and every player's Shapley value is $0$. Else, in any permutation of the big players, only one of them, the one that has rank $k$, makes $w_{\tau}$ change from $0$ to $1$. Every other player has contribution $0$ in that permutation. Since the player at rank $k$ is uniformly distributed over the $B(\tau)$ many big players, each big player's Shapley value is $\frac{1}{B(\tau)}$, and each small player's value is $0$. In particular, let $\mathbf{1}_{\{\ell_i(m)>\tau\}}$ and $\mathbf{1}_{\{B(\tau)\ge k\}}$ be indicator functions on the sets $\{\tau\,|\, \ell_i(m)>\tau\}$ and $\{\tau\,|\, B(\tau)\ge k\}$, respectively. Then, from the observations made so far, we have
		
		\[\varphi_i(w_{\tau})=\dfrac{\mathbf{1}_{\{\ell_i(m)>\tau\}}\cdot \mathbf{1}_{\{B(\tau)\ge k\}} }{B(\tau)}.\]
		
		By the linearity axiom of Shapley values we thus have
		
		\[\varphi_i(v_{(k,m)})=\int_{0}^{\infty}\varphi_i(w_{\tau})d\tau=\int_{0}^{\ell_i(m)}\dfrac{\mathbf{1}_{\{B(\tau)\ge k\}}}{B(\tau)}d\tau.\]

		Observe that $B(\tau)\ge k$ if and only if $\tau<t_{(k,m)}$ and $B(\tau)=j$ for $\tau \in [t_{(j+1,m)},t_{(j,m)})$. Therefore, the integral runs from $0$ up to $\min(\ell_i(m),t_{(k,m)})=t_{(\max(r_i,k),m)}$. Furthermore, $\dfrac{\mathbf{1}_{\{B(\tau)\ge k\}}}{B(\tau)}$ is a step function and splitting the integral at the sorted endpoints gives exactly
		\[\varphi_i(v_{(k,m)})=\sum_{j=\max (r_i,k)}^N\dfrac{t_{(j,m)}-t_{(j+1,m)}}{j}. \qedhere\]
		\end{proof}
		
		When $k=1$, $v_{(1,m)}=\max_{i\in S}\ell_i(m)$, which is (up to reversing cost and value) the airport game \cite{Littlechild1977}. Relabelling the players in descending order,
		\Cref{theorem:closed_form_landscapes} recovers the formula of Littlechild and Owen \cite{Littlechild1973}. Our proof, via the threshold games $w_{\tau}$, handles the general $k>1$ case.
		
	\begin{remark}
		Players ranked at or above the target layer, $r_i\le k$, all receive the same value $\varphi_i(v_{(k,m)})=\sum_{j=k}^{N}\dfrac{t_{(j,m)}-t_{(j+1,m)}}{j}$, i.e., they are symmetric contributors to reaching the $k$-th order statistic at the grid point $m$, regardless of how much their raw values differ. Players ranked below get a partial sum that is no larger, and strictly smaller whenever the omitted gaps are positive. Because the formula depends on rank and consecutive gaps instead of an index-based tiebreak, two persistence diagram points with nearly equal values automatically receive nearly equal attribution, i.e., there is no arbitrary winner. Furthermore, the credit allocation happens independently at every landscape coordinate $(k,m)$. In particular, a point that is for example rank $1$ at one grid location $m$ can be rank $5$ at another location. These values are only combined into a single value per persistence diagram point once we wish to allocate credit based on a model's prediction later in the manuscript.
	\end{remark}
	
	By \Cref{theorem:shap_unique}, the closed form in \Cref{theorem:closed_form_landscapes} is the only way  to fairly allocate credit to each player, in the game $v_{(k,m)}$.
	
	\begin{corollary}
	\label{corollary:uniqueness_landscape_games}
		For each ordered pair $(k,m)$, $\varphi(v_{(k,m)})$ is the only credit allocation rule satisfying Efficiency, Symmetry, Linearity and Null player axioms for the game $v_{(k,m)}$.
	\end{corollary}
	
	The following is a toy example illustrating the computation of Shapley values of the games $v_{(k,m)}$ using the closed form from \Cref{theorem:closed_form_landscapes}.
	
	\begin{example}
	\label{example:closed_form}
		Consider a  $3$-point persistence diagram $A=(1,5)$, $B=(2,6)$, $C=(3,4)$. The associated tent functions and the persistence landscapes are illustrated in \Cref{fig:1}.
		
		\begin{itemize}[left=0pt]
			\item For $k=1$, for $m\in (1,2]$ only the tent function $\ell_A$ is nonzero. Thus, player $A$ gets full credit, $\varphi_A(v_{(1,m)})=\lambda(1,m)=\ell_A(m)=m-1$ and $\varphi_B(v_{(1,m)})=\varphi_C(v_{(1,m)})=0$. Similarly, for $m\in [5,6)$ only the tent function $\ell_B$ is nonzero and thus $\varphi_B(v_{(1,m)})=\lambda(1,m)=\ell_B(m)=6-m$ and $\varphi_A(v_{(1,m)})=\varphi_C(v_{(1,m)})=0$.
			\item For $k=1$, for $m\in (2,3]$ only the tent functions $\ell_A$ and $\ell_B$ are nonzero, with $\ell_A(m)$ having rank $1$. Thus, by \Cref{theorem:closed_form_landscapes} we have 
			\[\varphi_A(v_{(1,m)})=\dfrac{\ell_A(m)-\ell_B(m)}{1}+\dfrac{\ell_B(m)}{2}=(m-1)-(m-2)+\dfrac{m-2}{2}=\dfrac{m}{2},\]
			and $\varphi_B(v_{(1,m)})=\dfrac{\ell_B(m)}{2}=\dfrac{m-2}{2}$, $\varphi_C(v_{(1,m)})=0$.
			\item For $k=1$, for $m\in (3,3.5]$ all three tent functions are nonzero. The values $\ell_C(m)$ always have rank $3$ and for $m\in (3,3.5)$ the values $\ell_A(m)$ have rank $1$ and the values $\ell_B(m)$ have rank $2$ (and are tied with $\ell_A(m)$ at $m=3.5$). Therefore, by \Cref{theorem:closed_form_landscapes}, we have 
			\begin{align*}
			\varphi_A(v_{(1,m)})&=\dfrac{\ell_A(m)-\ell_B(m)}{1}+\dfrac{\ell_B(m)-\ell_C(m)}{2}+\dfrac{\ell_C(m)}{3}\\
			&=(5-m)-(m-2)+\dfrac{(m-2)-(m-3)}{2}+\dfrac{m-3}{3}=\dfrac{13}{2}-\dfrac{5m}{3}
			\end{align*}
			\begin{align*}
			\varphi_B(v_{(1,m)})&=\dfrac{\ell_B(m)-\ell_C(m)}{2}+\dfrac{\ell_C(m)}{3}=\dfrac{(m-2)-(m-3)}{2}+\dfrac{m-3}{3}=\dfrac{m}{3}-\dfrac{1}{2},
			\end{align*}
			and $\varphi_C(v_{(1,m)})=\dfrac{\ell_C(m)}{3}=\dfrac{m-3}{3}$.
		\end{itemize}
		In similar fashion, we can compute the values $\varphi_A(v_{(k,m)})$, $\varphi_B(v_{(k,m)})$ and $\varphi_C(v_{(k,m)})$ for all choices of $k$ and $m$. We summarized them all in \cref{tab:toy-shapley}.
		
		\begin{table}[H]
\centering
\small
\renewcommand{\arraystretch}{1.35}
\begin{tabular}{|c|c|c|c|c|}
\hline
$k$ & $m$ & $\varphi_A(v_{(k,m)})$ & $\varphi_B(v_{(k,m)})$ & $\varphi_C(v_{(k,m)})$ \\
\hline\hline
1 & $1< m\le 2$   & $m-1$ & $0$ & $0$ \\
1 & $2\le m\le 3$ & $\frac{m}{2}$ & $\frac{m-2}{2}$ & $0$ \\
1 & $3\le m\le 3.5$ & $\frac{13}{2}-\frac{5m}{3}$ & $\frac{m}{3}-\frac12$ & $\frac{m-3}{3}$ \\
1 & $3.5\le m\le 4$ & $\frac{11}{6}-\frac{m}{3}$ & $\frac{5m}{3}-\frac{31}{6}$ & $\frac{4-m}{3}$ \\
1 & $4\le m\le 5$ & $\frac{5-m}{2}$ & $\frac{7-m}{2}$ & $0$ \\
1 & $5\le m<6$    & $0$ & $6-m$ & $0$ \\
\hline
2 & $2\le m\le 3$ & $\frac{m-2}{2}$ & $\frac{m-2}{2}$ & $0$ \\
2 & $3\le m\le 3.5$ & $\frac{m}{3}-\frac12$ & $\frac{m}{3}-\frac12$ & $\frac{m-3}{3}$ \\
2 & $3.5\le m\le 4$ & $\frac{11}{6}-\frac{m}{3}$ & $\frac{11}{6}-\frac{m}{3}$ & $\frac{4-m}{3}$ \\
2 & $4\le m\le 5$ & $\frac{5-m}{2}$ & $\frac{5-m}{2}$ & $0$ \\
\hline
3 & $3\le m\le 3.5$ & $\frac{m-3}{3}$ & $\frac{m-3}{3}$ & $\frac{m-3}{3}$ \\
3 & $3.5\le m\le 4$ & $\frac{4-m}{3}$ & $\frac{4-m}{3}$ & $\frac{4-m}{3}$ \\
\hline
\multicolumn{2}{|c|}{all other $(k,m)$} & $0$ & $0$ & $0$ \\
\hline
\end{tabular}
\caption{Shapley values of $A=(1,5)$, $B=(2,6)$, $C=(3,4)$ in each landscape game $v_{(k,m)}$. Each row's values sum to $\lambda_k(m)$ (efficiency).}
\label{tab:toy-shapley}
\end{table}
	\end{example}
	
	\begin{figure}[H]
		\centering
			\begin{tikzpicture}[scale=0.85]
				\draw[->] (0,0)--(6.5,0) node[right] {$b$ (birth)};
				\draw[->] (0,0)--(0,6.5) node[above] {$d$ (death)};
				\draw[dashed] (0,0)--(6.3,6.3);
				\filldraw[blue] (1,5) circle (2.5pt) node[left=2pt]{A};
				\filldraw[red] (2,6) circle (2.5pt) node[above=2pt]{B};
				\filldraw[green!55!black] (3,4) circle (2.5pt) node[below=1pt]{C};
				\node at (3.2,-1.1) {(a) persistence diagram};
				
				\begin{scope}[shift={(9,4)}]
					\draw[->] (0,0)--(6.5,0) node[right] {$m$};
					\draw[->] (0,0)--(0,3) node[above] {$\lambda(m)$};
					\draw[blue, line width=1.5pt] (1,0)--(3,2)--(5,0);
					\draw[red,line width=1.5pt] (2,0)--(4,2)--(6,0);
					\draw[green!55!black, line width=1.5pt] (3,0)--(3.5,0.5)--(4,0);
					\node[blue] at (1.75,1.35) {$\ell_A$};
					\node[red] at (5.25,1.35) {$\ell_B$};
					\node[green!55!black] at (3.5,0.85) {$\ell_C$};
					\foreach \x in {1,2,3,3.5,4,5,6}{
					  \draw[gray, dotted] (\x,0) -- (\x,3);
 				 \draw (\x,0) -- (\x,-0.06) node[below, font=\scriptsize] {$\x$};
						}
				\end{scope}
				
				\begin{scope}[shift={(9,0)}]
					\draw[->] (0,0)--(6.5,0) node[right] {$m$};
					\draw[->] (0,0)--(0,3) node[above] {$\lambda(m)$};
					\draw[RoyalBlue, line width=1.5pt] (1,0)--(3,2)--(3.5,1.5)--(4,2)--(6,0);
					\draw[Maroon,line width=1.5pt] (2,0)--(3.5,1.5)--(5,0);
					\draw[TealBlue, line width=1.5pt] (3,0)--(3.5,0.5)--(4,0);
					\node[RoyalBlue] at (1.75,1.35) {$\lambda_1$};
					\node[Maroon] at (2.6,1.15) {$\lambda_2$};
					\node[TealBlue] at (3.5,0.85) {$\lambda_3$};
					\node at (3.2,-1.1) {(b) persistence landscapes};
					\foreach \x in {1,2,3,3.5,4,5,6}{
					  \draw[gray, dotted] (\x,0) -- (\x,3);
 				 \draw (\x,0) -- (\x,-0.06) node[below, font=\scriptsize] {$\x$};
						}
				\end{scope}
			\end{tikzpicture}
		\caption{Persistence diagram and landscape example. (a) A persistence diagram with three points. (b) The associated tent functions $\ell_A$, $\ell_B$ and $\ell_C$ of points from (a) together with the three persistence landscape layers. }
		\label{fig:1}
	\end{figure}
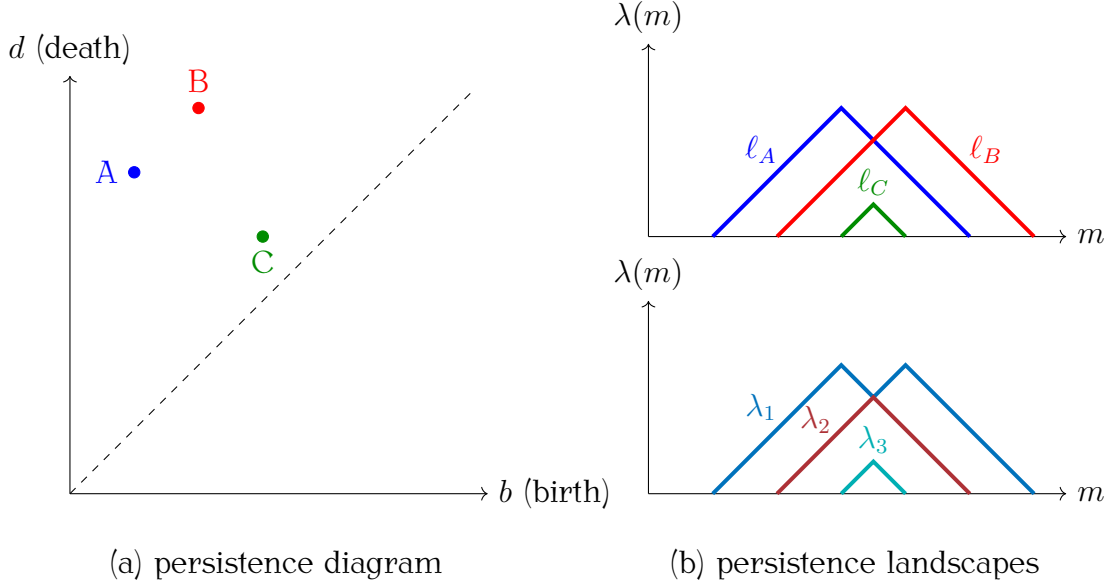
	
	From \Cref{theorem:closed_form_landscapes} we have the following.
	
	\begin{proposition}
		\label{prop:layer-sum}
		Fix $m\in\mathbb{R}$ and let $r_i$ be the rank of player $i$ among the tent values
		$\ell_1(m),\dots,\ell_N(m)$. Then for every $K\ge 1$,
		\[ \sum_{k=1}^{K}\varphi_i\bigl(v_{(k,m)}\bigr)\;=\; \sum_{j=r_i}^{N}\frac{\min(K,j)}{j}\,\bigl(t_{(j,m)}-t_{(j+1,m)}\bigr) \;\le\; \ell_i(m),\]
		with equality $ \sum_{k=1}^{K}\varphi_i\bigl(v_{(k,m)}\bigr)\;=\ell_i(m)$ if $K\ge N$. 
	\end{proposition}
	
	\begin{proof}
		Let $g_j\od t_{(j,m)}-t_{(j+1,m)}\ge 0$ for $j=1,\dots,N$, with $t_{(N+1,m)}=0$. By \Cref{theorem:closed_form_landscapes} we have
		$\varphi_i(v_{(k,m)})=\sum_{j=\max(r_i,k)}^{N} \dfrac{g_j}{j}$. Therefore we have
		\[\sum_{k=1}^{K}\varphi_i(v_{(k,m)})=\sum_{k=1}^{K}\sum_{j=\max(r_i,k)}^{N}\dfrac{g_j}{j}=\sum_{j=r_i}^{N}\dfrac{g_j}{j}\,\left|\{k\in\{1,\dots,K\}\,|\, \max(r_i,k)\le j\}\right|.\]
		For $j\ge r_i$, $\max(r_i,k)\le j$ is equivalent to $k\le j$, thus $\left|\{k\in\{1,\dots,K\}\,|\, \max(r_i,k)\le j\}\right|=\min(K,j)$ which gives the stated formula. Since $\dfrac{\min(K,j)}{j}\le 1$ and $g_j\ge 0$,
		\[\sum_{k=1}^{K}\varphi_i(v_{(k,m)})\le\sum_{j=r_i}^{N}g_j=\sum_{j=r_i}^N(t_{(j,m)}-t_{(j+1,m)})=t_{(r_i,m)}=\ell_i(m),\]
		since the sum telescopes and $t_{(N+1,m)}=0$. If $K\ge N$ then $\min(K,j)=j$ for all $j\le N$, so every inequality above is an equality.
	\end{proof}
	
	\begin{remark}
		\label{remark:decomposition}
		By efficiency, each landscape layer decomposes additively into pointwise components,
		\[ \lambda_k(m)=\sum_{i=1}^{N}\varphi_i\bigl(v_{(k,m)}\bigr),\]
		and each component $m\mapsto\varphi_i(v_{(k,m)})$ is a nonnegative, piecewise linear
		function of $m$ (\Cref{prop:layer-sum} and \Cref{theorem:closed_form_landscapes}).
		\Cref{prop:layer-sum} shows that, when $K\ge N$, summing the components of
		point $i$ over all layers recovers its tent function $\ell_i(m)$. Thus, LandscapeSHAP is precisely a way to redistribute the tent function of each point across
		the layers $k$ in a fair manner that respects the order-statistic structure of the
		landscape. Unlike a Fourier expansion, the components are data dependent and are not
		orthogonal, or a basis and the decomposition is a partition of each layer among the diagram
		points. \Cref{fig:toy-decomposition} illustrates this for \Cref{example:closed_form}.
	\end{remark}
	
	\begin{figure}[h]
\centering
\begin{tikzpicture}
\begin{groupplot}[
    group style={group size=2 by 2, horizontal sep=1.3cm, vertical sep=1.8cm},
    width=7.4cm, height=4.6cm,
    xmin=0, xmax=7, ymin=0, ymax=2.3,
    xtick={1,2,3,3.5,4,5,6}, xmajorgrids,
    major grid style={dotted, gray},
    tick label style={font=\footnotesize},
    title style={font=\small},
    xlabel={$m$},
  ]

\nextgroupplot[title={$k=1$: $\lambda_1(m)=\sum_i\varphi_i(v_{(1,m)})$},
    area style, stack plots=y,
    legend style={at={(0.02,0.98)}, anchor=north west, draw=none, font=\footnotesize}]
\addplot[fill=blue!70, draw=none]  coordinates {(0,0)(1,0)(2,1)(3,1.5)(3.5,0.6667)(4,0.5)(5,0)(6,0)(7,0)} \closedcycle;
\addplot[fill=red!70, draw=none]   coordinates {(0,0)(1,0)(2,0)(3,0.5)(3.5,0.6667)(4,1.5)(5,1)(6,0)(7,0)} \closedcycle;
\addplot[fill=green!60!black, draw=none] coordinates {(0,0)(1,0)(2,0)(3,0)(3.5,0.1667)(4,0)(5,0)(6,0)(7,0)} \closedcycle;
\addplot[stack plots=false, black, thick, forget plot] coordinates {(0,0)(1,0)(2,1)(3,2)(3.5,1.5)(4,2)(5,1)(6,0)(7,0)};
\legend{$\varphi_A$,$\varphi_B$,$\varphi_C$}

\nextgroupplot[title={$k=2$: $\lambda_2(m)=\sum_i\varphi_i(v_{(2,m)})$},
    area style, stack plots=y]
\addplot[fill=blue!70, draw=none]  coordinates {(0,0)(1,0)(2,0)(3,0.5)(3.5,0.6667)(4,0.5)(5,0)(6,0)(7,0)} \closedcycle;
\addplot[fill=red!70, draw=none]   coordinates {(0,0)(1,0)(2,0)(3,0.5)(3.5,0.6667)(4,0.5)(5,0)(6,0)(7,0)} \closedcycle;
\addplot[fill=green!60!black, draw=none] coordinates {(0,0)(1,0)(2,0)(3,0)(3.5,0.1667)(4,0)(5,0)(6,0)(7,0)} \closedcycle;
\addplot[stack plots=false, black, thick] coordinates {(0,0)(1,0)(2,0)(3,1)(3.5,1.5)(4,1)(5,0)(6,0)(7,0)};

\nextgroupplot[title={$k=3$: $\lambda_3(m)=\sum_i\varphi_i(v_{(3,m)})$},
    area style, stack plots=y]
\addplot[fill=blue!70, draw=none]  coordinates {(0,0)(1,0)(2,0)(3,0)(3.5,0.1667)(4,0)(5,0)(6,0)(7,0)} \closedcycle;
\addplot[fill=red!70, draw=none]   coordinates {(0,0)(1,0)(2,0)(3,0)(3.5,0.1667)(4,0)(5,0)(6,0)(7,0)} \closedcycle;
\addplot[fill=green!60!black, draw=none] coordinates {(0,0)(1,0)(2,0)(3,0)(3.5,0.1667)(4,0)(5,0)(6,0)(7,0)} \closedcycle;
\addplot[stack plots=false, black, thick] coordinates {(0,0)(1,0)(2,0)(3,0)(3.5,0.5)(4,0)(5,0)(6,0)(7,0)};

\nextgroupplot[title={$\sum_k\varphi_i(v_{(k,m)})=\ell_i(m)$ (dashed)}]
\addplot[fill=blue!70,  fill opacity=0.35, draw=blue!70,  very thick] coordinates {(0,0)(1,0)(2,1)(3,2)(3.5,1.5)(4,1)(5,0)(6,0)(7,0)} \closedcycle;
\addplot[fill=red!70,   fill opacity=0.35, draw=red!70,   very thick] coordinates {(0,0)(1,0)(2,0)(3,1)(3.5,1.5)(4,2)(5,1)(6,0)(7,0)} \closedcycle;
\addplot[fill=green!60!black, fill opacity=0.35, draw=green!60!black, very thick] coordinates {(0,0)(1,0)(2,0)(3,0)(3.5,0.5)(4,0)(5,0)(6,0)(7,0)} \closedcycle;
\addplot[black, dashed, thin] coordinates {(0,0)(1,0)(2,1)(3,2)(3.5,1.5)(4,1)(5,0)(6,0)(7,0)};
\addplot[black, dashed, thin] coordinates {(0,0)(1,0)(2,0)(3,1)(3.5,1.5)(4,2)(5,1)(6,0)(7,0)};
\addplot[black, dashed, thin] coordinates {(0,0)(1,0)(2,0)(3,0)(3.5,0.5)(4,0)(5,0)(6,0)(7,0)};

\end{groupplot}
\end{tikzpicture}
\caption{Shapley decomposition of the landscape in \Cref{example:closed_form}. Panels 1--3: each layer $\lambda_k$ (black) is the sum of the players' values $\varphi_i(v_{(k,m)})$ (stacked: $A$ blue, $B$ red, $C$ green). Panel 4: summing over layers recovers each tent function $\ell_i$ (dashed).}
\label{fig:toy-decomposition}
\end{figure}
	
	\subsection{Monotonicity}
	
	Recall that a cooperative game $v$ is \emph{monotone} if for every pair of coalitions $S,T$ with $S\subset T$, we have $v(S)\le v(T)$.
	
	\begin{proposition}
		For every pair $(k,m)$, the game $v_{(k,m)}$ is monotone. Furthermore, $\varphi_i(v_{(k,m)})\ge 0$ for every player $i$. 
	\end{proposition}
	
	\begin{proof}
		For $\tau\ge 0$, recall the game $w_{\tau}(S)$ from the proof of \Cref{theorem:closed_form_landscapes}. Recall that $v_{(k,m)}(S)=\int_0^{\infty}w_{\tau}(S)d\tau$. If $S\subset T$, it follows from definition that $w_{\tau}(S)\le w_{\tau}(T)$ and therefore $v_{(k,m)}(S)\le v_{(k,m)}(T)$. Therefore, the game $v_{(k,m)}$ is monotone. In particular, if $i$ is any player we have that the marginal contribution $v_{(k,m)}(S\cup\{i\})-v_{(k,m)}(S)\ge 0$ for any coalition $S$ and thus every Shapley value is nonnegative.
	\end{proof}
	
	\begin{remark}
		Later on we will see that Shapley values of diagram points can be negative with respect to a linear model trained on landscapes. By the previous result, the Shapley values of  the landscape games can never be negative. This means that any negative Shapley value of persistence diagram point comes directly from the model's coefficients and not from the underlying organization of the persistence diagram.
	\end{remark}
	
	\subsection{Stability}
	\label{section:stability_landscape_game}
	Given two persistence diagrams, $D$ and $D'$, a \emph{matching} between them is a bijection $\gamma$ between $D\cup \Delta$ and $D'\cup \Delta$ where $\Delta=\{(x,x)\, |\, x\in \mathbb{R}\}$ is the diagonal taken with infinite multiplicity. In particular, every off diagonal point in $D$ is either matched with an off diagonal point in $D'$ or implicitly matched with its nearest point in $\Delta$, and similarly for points in $D'$. This is what allows for persistence diagrams with different number of points to be compared. The \emph{bottleneck distance} is the infimum, over all such matchings $\gamma$, of the largest cost $||(b,d)-\gamma(b,d)||_{\infty}$ any matched pair of persistence diagram points $(b,d)$ undergoes. In particular,
	\[d_B(D,D')\od\inf_{\gamma}||(b,d)-\gamma(b,d)||_{\infty},\]
	where $\gamma$ is a matching between $D$ and $D'$.
	The classical stability theorem for persistence diagrams \cite{CohenSteiner2006} states that if the data is perturbed by $\varepsilon$, the bottleneck distance between the corresponding persistence diagrams does not change by more than $\varepsilon$. Persistence landscapes also have this stability directly \cite{bubenik2015statistical}. In particular, the persistence landscape map is Lipschitz with respect to $d_B$, so a small bottleneck perturbation of the diagram produces a correspondingly small change in every landscape coordinate. Therefore, it is natural to ask if there are similar stability results for the Shapley values of persistence diagram points. We answer that question positively with \Cref{theorem:stability_for_landscape_games} below.

	We may assume that persistence diagrams $D$ and $D'$ have an equal number of points $N$ with a fixed correspondence between them. Let $\gamma$ be any matching (e.g., a $d_B$-optimal one, if $\varepsilon$ below is to be taken equal to $d_B(D,D')$) and pad whichever diagram has fewer off diagonal points with copies of its matched diagonal points, so that $\gamma$ becomes an honest bijection from points in $D$ to points in $D'$.  Note that for a diagonal point, $(x,x)$, the tent function is identically $0$, $\ell (m)=0$ for every $m$, making it a null player in every game $v_{(k,m)}$ and hence the assigned Shapley value is $0$ by \Cref{theorem:closed_form_landscapes}. Therefore, introducing diagonal points to equalize the sizes between diagrams never changes any ``real" player’s Shapley value.

	Let $D,D'$ be two diagrams indexed by the same set $\{1,\dots ,N\}$ via a fixed bottleneck distance optimal matching between them, per the discussion above. For simplicity, let point $i$ of $D$ and point $i$ of $D'$ denote the two ends of a single matched pair, $(b_i,d_i) \in D$ and $(b_i',d_i')\in D'$, not merely the $i$-th point of each diagram under some unrelated ordering. Let $\ell_i(m)$ and $\ell'_i(m)$ be the tent functions of the point $i$ in $D$ and $D'$, respectively.  Let $\varphi_r(v_{(k,m)}(D))$ and $\varphi_r(v_{(k,m)}(D'))$ be the Shapley values of a player at rank $r$, at the coordinate $(k,m)$, for the games $v_{(k,m)}$ constructed from the diagrams $D$ and $D'$, respectively. Then we have the following result.
	
	\begin{theorem}
	\label{theorem:stability_for_landscape_games}
		For every rank $r$, we have the following inequality
		\[|\varphi_r(v_{(k,m)}(D))-\varphi_r(v_{(k,m)}(D'))|\le \left(\dfrac{2}{\max(r,k)}-\dfrac{1}{N}\right)\varepsilon\le 2\varepsilon,\]
		where $\varepsilon=\underset{i}{\max}\max(|b_i-b_i'|,|d_i-d_i'|)$, i.e., $\varepsilon=d_B(D,D')$. 
	\end{theorem}
	
	\begin{proof}
		Consider a player $i$, and let $A=\max(0,m-b_i)$, $A'=\max(0,m-b_i')$, $B=\max(0,d_i-m)$ and $B'=\max(0,d_i'-m)$. Then, by definition, $\ell_i(m)=\min (A,B)$ and $\ell_i'(m)=\min(A',B')$.
		The mappings $b\mapsto \max(0,m-b)$ and $d\mapsto \max(0,d-m)$ are $1$-Lipschitz in $b$ and $d$, respectively. Therefore, $|A-A'|\le |b_i-b'_i|$ and $|B-B'|\le |d_i-d_i'|$. Now consider all the following possible cases:
		\begin{enumerate}[left=0pt]
			\item Suppose that $\min(A,B)\ge \min (A',B')$ and that $\min(A',B')=A'$. Then, $\min (A,B)\le A\le A'+|A-A'|=\min(A',B')+|A-A'|$. 
			\item Suppose that $\min(A,B)\ge \min (A',B')$ and that $\min(A',B')=B'$. Then, $\min (A,B)\le B\le B'+|B-B'|=\min(A',B')+|B-B'|$. 
			\item Suppose that $\min(A,B)\le \min (A',B')$ and that $\min(A,B)=A$. Then, $\min (A',B')\le A'\le A+|A-A'|=\min(A,B)+|A-A'|$. 
			\item Suppose that $\min(A,B)\le \min (A',B')$ and that $\min(A,B)=B$. Then, $\min (A',B')\le B'\le B+|B-B'|=\min(A,B)+|B-B'|$. 
		\end{enumerate}
		From all these cases we conclude that  $\min$ is $1$-Lipschitz in its two arguments with respect to the sup norm, i.e., we have $|\min(A,B)-\min(A',B')|\le \max(|A-A'|,|B-B'|)$. Therefore, we have the inequality
		
		\[|\ell_i(m)-\ell_i'(m)|\le \max(|A-A'|,|B-B'|)\le \max(|b_i-b_i'|,|d_i-d_i'|)\le \varepsilon,\]
		for every player $i$. 
		
		Let $t_{(1,m)}\ge t_{(2,m)}\ge \cdots\ge t_{(N,m)}$ be the sorting of values $\ell_i(m)$, for $i=1,\dots, N$, in decreasing order and set $t_{(N+1,m)}=0$. Similarly, let $t'_{(1,m)}\ge t'_{(2,m)}\ge \cdots\ge t'_{(N,m)}$ be the sorting of values $\ell_i'(m)$, for $i=1,\dots, N$, in decreasing order and set $t'_{(N+1,m)}=0$. Now let $j$ be a fixed rank and let $C$ be the set of players attaining $j$ many largest values of $\ell_i(m)$, i.e., $t_{(j,m)}=\min_{i\in C}\ell_i(m)$. For every $i\in C$, $\ell_i'(m)\ge \ell_i(m)-\varepsilon\ge t_{(j,m)}-\varepsilon$, hence at least $j$ many of the $\ell_i'(m)$ satisfy $\ell_i'(m)\ge t_{(j,m)}-\varepsilon$. This implies that $t'_{(j,m)}\ge t_{(j,m)}-\varepsilon$. Similarly, we also have that $t_{(j,m)}\ge t_{(j,m)}'-\varepsilon$. Therefore, $|t_{(j,m)}-t_{(j,m)}'|\le \varepsilon$ for every $j$.
		
		Let $a=\max(r,k)$. Then, by \Cref{theorem:closed_form_landscapes}, we can write 
		\[\varphi_r(v_{(k,m)}(D))=\sum_{j=a}^Nc_jt_{(j,m)},\]
		where $c_a=\dfrac{1}{a}$, and $c_j=\dfrac{1}{j}-\dfrac{1}{j-1}$ for $j>a$. The coefficients $c_j$ telescope: 
		\[\sum_{j=a}^N|c_j|=\dfrac{1}{a}+\sum_{j=a+1}^N\left(\dfrac{1}{j-1}-\dfrac{1}{j}\right)=\dfrac{2}{a}-\dfrac{1}{N}.\]
		Therefore, we have the desired inequality
		\[|\varphi_r(v_{(k,m)}(D))-\varphi_r(v_{(k,m)}(D'))|=\left|\sum_jc_j(t_{(j)}-t'_{(j)})\right| \le \left(\sum_j |c_j|\right)\varepsilon=\left(\dfrac{2}{a}-\dfrac{1}{N}\right)\varepsilon.\qedhere\]
	\end{proof}
	
	\subsection{Computation}
	Here we give an algorithm that computes the values $\varphi_i(v_{(k,m)})$ from \Cref{theorem:closed_form_landscapes} for a given point $m$.
	
	\begin{algorithm}[H]
	\caption{LandscapeSHAP(m)}
	\label{alg:grid_shap}
	\textbf{Input:} Diagram points with tent values $\ell_i(m)$ \\
	\textbf{Output:} Shapley value $\varphi_i(v_{(k,m)})$ for every player $i$ and layer $k$ at the point $m$.
		\begin{algorithmic}[1]
		\State Sort the active tent values in descending order $t_{(1,m)}\ge \cdots \ge t_{(N,m)}$ and record each player’s rank $r_i$. This has cost $O(N\log N)$.
		\State Build an array $T$ to store the sums from \Cref{theorem:closed_form_landscapes}:
				\[T_{N+1}=0,\quad  T_j=T_{j+1}+\dfrac{t_{(j,m)}-t_{(j+1,m)}}{j},\quad j=N,\dots, 1\]
		\State Compute $\varphi_i(v_{(k,m)})$ in a single $O(1)$ lookup.
				\[\varphi_i(v_{(k,m)})=T_{\max(r_i,k)}.\]
		\State \textbf{return} the array of Shapley values $\varphi_i(v_{(k,m)})$.
		\end{algorithmic}
	\end{algorithm}
		
	\section{Linear models on landscapes}
	
	\label{section:linear_models}
	
	Here we consider a linear model trained on persistence landscapes. We don't require the model to be well-trained, regularized or pre-processed in any way. We are explaining how any given model makes predictions and where the model comes from or its performance metrics on data do not affect our method. The only requirement is that any transformation of the raw landscape values into a prediction is a fixed affine map. In other words, the coefficients of the model are frozen after training and do not depend on which coalition $S$ is being evaluated.
	
	To have a landscape be an input to such a model it is necessary to discretize the landscape function, that is select a finite set of ``grid" points $M\subset \mathbb{R}$ and consider the restriction of the landscape function
	\[\lambda: \{1,\dots, K\}\times M\to \mathbb{R}.\]
	
	 From now on, in the rest of the manuscript we assume every landscape is a function with a discrete domain such as above. The linear model has the general form
	
	\[f(\lambda)=\sum_{k,m}w_{(k,m)}\lambda(k,m)+b,\]
	
	where $(k,m)\in \{1,\dots, K\}\times M$ ranges over the landscape coordinates, and $w_{(k,m)}$, $b$ are the weights and bias of the model, respectively. The bias $b$ does not affect the Shapley values, since it shifts $v(S)$ uniformly for any coalition $S$ and it cancels out in every marginal contribution.	 Let $\Lambda$ be the feature map from persistence diagrams to landscapes. Let $f$ be a linear model that takes as input vectors of dimension $K\times |M|$. We define the cooperative game on coalitions of persistence diagram points $S$ by
		\[v_f(S)\od f(\Lambda(S))-f(\Lambda(\varnothing)),\quad S\subset \{1,\dots, N\}.\]

		We wish to compute the Shapley values of the game $v_f$, i.e. $\varphi_i(v_f)$.
	
	\subsection{Exactness and uniqueness}
	
	Note that the $(k,m)$ coordinate of the landscape $\Lambda(S)$ is simply the value of the landscape game $v_{(k,m)}(S)$, $\Lambda(S)_{(k,m)}=v_{(k,m)}(S)$.
	Because the Shapley values are linear in the characteristic functions and by \Cref{theorem:closed_form_landscapes}, we thus immediately have the following.
	
	\begin{theorem}
	\label{theorem:closed_form_linear_model}
	Let $f$ be a linear model. The Shapley value of player $i$ in the game $v_f$ is
	\[\varphi_i (v_f)=\sum_{k,m}w_{(k,m)}\varphi_i(v_{(k,m)}).\]
	\end{theorem}
	
	Additionally, from \Cref{theorem:shap_unique} we also have the following uniqueness result.

	\begin{corollary}
	\label{corollary:uniqueness_landscape_shap}
	For a linear model $f$ and the characteristic function $v_f$, any persistence diagram point credit attribution rule $\psi(v_f)$ satisfying Efficiency, Symmetry, Null player and Linearity for $v_f$ must equal $\varphi(v_f)$, for every diagram point.
	\end{corollary}
		
	\begin{remark}[Choice of the baseline]
	\label{remark:choice_of_baseline}
	This definition of $v_f$ assumes $v_f(\varnothing)=0$. Alternatively, we could have also defined the game as $v_f(S)\od f(\Lambda(S))$ and in this case $v_f(\varnothing)=b$. Both of these are different baselines from the standard convention of comparing against the dataset average features \cite{NIPS2017_8a20a862}. For crediting persistent homology classes, the empty persistence diagram baseline is arguably the more natural one since  an empty diagram truly has no topology. If a dataset average baseline is desired instead, the game must be redefined relative to that reference, and the argument above no longer collapses to a closed form as directly. This is a choice a TDA practitioner must make. 
	\end{remark}
	
	There is also a stability theorem for Shapley values of games $v_f$ for any model $f$ that we prove later (\Cref{theorem:stability_nonlinear_model}).
	
	\subsection{Computation}
	Note that for a point $m$, the players with $\ell_i(m)=0$ are always null in the game $v_{(k,m)}$ so they can be discarded at the point $m$. Note that they can be active at other points beside $m$. We give an algorithm for computing the Shapley values from \Cref{theorem:closed_form_linear_model}.
	
	\begin{algorithm}[H]
	\caption{LandscapeSHAP for linear models}
	\label{alg:linear_model}
	\textbf{Input:} Diagram points with tent values $\ell_i(m)$, linear model $f$ with model weights $w_{(k,m)}$ \\
	\textbf{Output:} Shapley value $\varphi_i(v_f)$ for every player (persistence diagram point) $i$.
		\begin{algorithmic}
		\State \text{Initialize} $\varphi_i\leftarrow 0$ for all $i$.
		\State \textbf{for} each grid point $m=1,\dots, M$: 
		\begin{itemize}
		\item[(a)] restrict to the active set $A=\{i\,|\, \ell_i(m)>0\}$ and run \text{LandscapeSHAP}(m) on it, which returns the values $\varphi_i(v_{(k,m)})$
		\item[(b)] \textbf{for} each layer $k=1,\dots, K$, \textbf{for} each $i\in A$:
				\[\varphi_i\leftarrow \varphi_i+w_{(k,m)}\cdot \varphi_i(v_{(k,m)})\]
		\end{itemize}
		\State \textbf{return} the array $\varphi$.
		\end{algorithmic}
	\end{algorithm}
	
	For $N$ on the order of a few hundred points, $K=40$ layers and $|M|=100$ grid points, this is on the order of $|M|\cdot N\log N$ for sorting ($\sim 10^5$ operations) plus $|M|\cdot N\cdot K$ for the lookups ($\sim 10^6$ operations). This is trivial for a modern machine.
	
	\subsection{Credit allocation for an averaged landscape}
	\label{section:average_landscapes}
	
	The Strong Law of Large Numbers for persistence landscapes (\cite[Theorem 9]{bubenik2015statistical}) makes it so that the average of landscapes over a random process is usually computed in applications. Let $D_1,\dots, D_n$ be persistence diagrams (e.g., obtained from point clouds that were generated i.i.d. as random variables) with corresponding landscapes $\lambda_1,\dots, \lambda_n$. Then, we can train or evaluate a model $f$ on a linear combination
	\[\overline{\lambda}=\sum_{j=1}^n\alpha_j \lambda_j,\] 
	where the choice  $\alpha_j=\dfrac{1}{n}$ gives the average landscape.  A natural question to ask is whether the persistence diagram point-level attribution of \Cref{theorem:closed_form_linear_model} still works, namely, can the credit allocation for the model's prediction on $\overline{\lambda}$ be pulled back to the individual points within each of the $n$ underlying persistence diagrams?
	
	Consider the disjoint union of all players $N=D_1\sqcup \cdots \sqcup D_n$ and the game 
	\[v_f(S)=f\left(\sum_{j=1}^n\alpha_j \Lambda(S\cap D_j)\right)-f\left(\Lambda(\varnothing)\right),\quad S\subset N.\]
	In other words, the value of the coalition $S$ is the model's prediction on the weighted sum of the landscapes obtained from the coalitions $S\cap D_j$, for each persistence diagram $D_j$. Then we have the following
	
	\begin{theorem}
	\label{theorem:average_landscapes}
		Let $f$ be a linear model on landscapes. For every $j$ and every player $i\in D_j$, 
		\[\varphi_i(v_f)=\alpha_j\cdot \varphi_i(v_f;D_j),\]
		where $\varphi_i(v_f;D_j)$ is the Shapley value of player $i$ from \Cref{theorem:closed_form_linear_model}, computed for the persistence landscape of the diagram $D_j$ in isolation, using the same weights $w_{(k,m)}$. In particular, no cross-diagram interaction terms appear.
	\end{theorem}
	
	\begin{proof}
		Because $f$ is linear, we have
		\[v_f(S)=\sum_{j=1}^n\alpha_j f(\Lambda(S\cap D_j))+b(\sum_j\alpha_j-1)-f(\Lambda(\varnothing))=\sum_{j=1}^n v_{(f,j)}(S)+\text{const},\]
		where $v_{(f,j)}(S)=\alpha_jf(\Lambda(S\cap D_j))$.
		Note that each game $v_{(f,j)}$ depends on $S$ only through the restriction to $S\cap D_j$. Let $j$ be fixed. For every player $i\not \in D_j$ and every coalition $S\subset N\setminus \{i\}$, we have $(S\cup\{i\})\cap D_j=S\cap D_j$. Therefore, $v_{(f,j)}(S\cup \{i\})=v_{(f,j)}(S)$ and thus every player outside $D_j$ is null in the game $v_{(f,j)}$. Furthermore, the additive constant does not affect the Shapley values. By linearity, we have $\varphi(v_f)=\sum_{j=1}^n\varphi(v_{(f,j)})$ so we only need to compute $\varphi_i(v_{(f,j)})$ for $i\in D_j$.
	Furthermore, for $i\in D_j$, we have 
	\[v_{(f,j)}(S\cup\{i\})-v_{(f,j)}(S)=v_{(f,j)}((S\cap D_j)\cup \{i\})-v_{(f,j)}(S\cap D_j)=\alpha_j(f(\Lambda((S\cap D_j)\cup \{i\}))-f(\Lambda(S\cap D_j))).\]
	Therefore, $\varphi_i(v_{(f,j)})=\alpha_j\varphi_i(v_f;D_j)$.
	\end{proof}
	
	\begin{remark}
		The result above depends entirely on $f$ being linear. For a nonlinear $f$, $v_f(S)$ generally cannot be written as a sum of games, each depending on only one diagram's points. In this case, $\varphi_i(v_f)$ must be estimated over the total player set $N$. This estimation will have much higher exponential cost in $|N|=\sum_{j}|D_j|$.
	\end{remark}

	\section{Nonlinear models on persistence landscapes}
	\label{section:nonlinear_models}
		
		Let $\Lambda$ as before be the feature map from persistence diagrams to landscapes. Let $f$ be a nonlinear model that takes as input vectors  of dimension $K\times |M|$. We define the cooperative game on coalitions of persistence diagram points $S$ by
		\[v_f(S)\od f(\Lambda(S))-f(\Lambda(\varnothing)),\quad S\subset \{1,\dots, N\}.\]

		Note that the $(k,m)$ coordinate of the landscape $\Lambda(S)$ is simply the landscape game $v_{(k,m)}(S)$, $\Lambda(S)_{(k,m)}=v_{(k,m)}(S)$. We wish to compute the Shapley values of the game $v_f$, $\varphi_i(v_f)$. They are defined by either of the averaging formulas (\Cref{def:shap_value,def:shap_value_permutation}).
		
		\subsection{Uniqueness}	

		The uniqueness result in (\Cref{corollary:uniqueness_landscape_shap}) does not depend on the linearity of a model, and by \Cref{theorem:shap_unique} we also have the following:
		
		\begin{corollary}
		\label{corollary:fairness_axioms}
			For any model $f$ there is exactly one credit allocation rule $\varphi(v_f)$ satisfying Efficiency, Symmetry, Linearity and Null player for the game $v_f$.
		\end{corollary}
		
		What the linearity of a model gets us is an exact closed form solution (\Cref{theorem:closed_form_linear_model}) and for a general nonlinear model no such exact formula is available. Furthermore, just like in \Cref{remark:choice_of_baseline}, the uniqueness in \Cref{corollary:fairness_axioms} is relative to a baseline.
	
	\subsection{Approximation}
	
	The game $v_f$ must be evaluated for every coalition $S$ and the Shapley values are then computed exactly only by averaging marginal contributions over all $2^N$ coalitions, which is not computable beyond small $N$. A common approach is to instead use Monte Carlo sampling of coalitions. This gives an approximation that converges to the exact Shapley values as the number of samples grows. Naive permutation sampling recomputes $\Lambda(S)$ from scratch at every step, of every permutation, and sorting the whole coalition $S$ at each of the grid points in $M$ costs $O(|S|\log |S|)$, and summing this over the $N$ steps of one permutation costs $O(|M| N^2 \log N)$.  
	
	\begin{algorithm}[H]
	\caption{Efficient Monte Carlo point-level attribution (any model)}
	\label{alg:nonlinear_model}
	\textbf{Input:} Diagram points with tent values $\ell_i(m)$, a model $f$, a number of permutations $n$ \\
	\textbf{Output:} An estimate $\hat{\varphi}_i(v_f)$ for every player $i$.
		\begin{algorithmic}
		\State \text{Initialize} $\hat{\varphi}_i(v_f)\leftarrow 0$ for all $i$.
		\State \textbf{for} $j\in 1,\dots, n$: 
		\State (a) sample uniformly a permutation $\pi$ of $\{1,\dots, N\}$; for each grid point $m$ initialize an empty array $B_m$ of size $K$; set $v_{\text{old}}\leftarrow f(\varnothing)$.
		\State (b) \textbf{for} each player $i$ in the order given by $\pi$:
		\begin{enumerate}
			\item[(i)] for each grid point $m$: append $\ell_i(m)$ to $B_m$ (size $K \to K+1$), re-sort the array, and keep only the top $K$ entries (cost $O(K\log K)$, independent
of how many points have joined so far)
			\item[(ii)] assemble the resulting $K\times |M|$ landscape matrix from the arrays $B_1,\dots ,B_M$ and evaluate $v_{\text{new}}\leftarrow f(\text{landscape})$;
			\item[(iii)] $\hat{\varphi}_i\leftarrow \hat{\varphi}_i+(v_{\text{new}}-v_{\text{old}}$); $v_{\text{old}}\leftarrow v_{\text{new}}$.  
		\end{enumerate}
		\State \textbf{return} $\hat{\varphi}_i(v_f)\leftarrow \dfrac{\hat{\varphi}_i}{n}$.
		\end{algorithmic}
	\end{algorithm}
	
	The only change from the naive sampling is step $(b)(i)$. Maintaining each grid point’s top $K$ tent values under insertion costs $O(K\log K)$ per grid point  instead of
	 $O(|S|\log |S|)$. Therefore, the cost of one permutation is $O(N|M|K\log K)$, which is linear rather than quadratic in $N$, since $K$ is chosen independently of the
	diagram size. This speeds up the featurization step $S \mapsto \Lambda(S)$ itself and so benefits step $(b)(ii)$ regardless of the model $f$. Evaluating $f$ on the resulting $K\times |M|$ matrix landscape is a separate, model-dependent cost on top of it. 	 
	
	 \subsection{Rate of convergence}
	 
	 \Cref{alg:nonlinear_model} is unbiased for $\varphi_i(f)$ regardless of $n$, but its variance, and hence the number of permutations needed for a target accuracy, is not yet known. A priori the accuracy could degrade as the diagram grows since a coalition $S$ ranges over an exponentially large space. Thus, one might expect the marginal contributions of a single player to become more erratic as $N$ grows. For the landscape featurization specifically, this does not happen. We show that inserting one player into a coalition can only move the landscape by an amount bounded by that point’s own tent function, no matter how large the coalition already is.
	 
	 \begin{lemma}
	 \label{lemma:bound_on_insertion_of_player}
	 	Let $S\subset \{1,\dots, N\}$ be a coalition of players and let $i_0\not \in S$. Let $t_{(1,m)}\ge t_{(2,m)}\ge \cdots \ge t_{(n,m)}$ be a sorting of $\ell_i(m)$ in decreasing order, for $i\in S$, where $n=|S|$. Then
		\[\sum_{k=1}^K |v_{(k,m)}(S\cup \{i_0\})-v_{(k,m)}(S)|\le \ell_{i_0}(m),\]
		with equality when $K\ge n+1$.
	 \end{lemma}
	 
	 \begin{proof}
	 	Let $r\in \{1,\dots, n+1\}$ be the rank $\ell_{i_0}(m)$ assumes once inserted into the coalition $S$. Then $t_{(r-1,m)}\ge \ell_{i_0}(m) \ge t_{(r,m)}$, where we set $t_{(0,m)}=\infty$ and $t_{(n+1,m)}=0$. In particular, inserting $i_0$ leaves the layers $k<r$ unchanged. More specifically, layer $r$'s new value is $\ell_{i_0}(m)$ instead of the old $t_{(r,m)}$. Because the values are sorted in decreasing order we have that 
		\[\ell_{i_0}(m)-t_{(r,m)}\ge 0, \quad t_{(j-1,m)}-t_{(j,m)}\ge 0, r<j\le n.\]
		Therefore, since $t_{(j-1,m)}-t_{(j,m)}$ fit into a telescoping sum, we have 
		\[(\ell_{i_0}(m)-t_{(r,m)})+\sum_{j=r+1}^n t_{(j-1,m)}-t_{(j,m)}+t_{(n,m)}=(\ell_{i_0}(m)-t_{(r,m)})+(t_{(r,m)}-t_{(n,m)})+t_{(n,m)}=\ell_{i_0}(m),\]
		for $K\ge n+1$. If $K<n+1$, the sum above is truncated. Its last nonnegative terms are dropped, so it decrease, giving an inequality in this case.
	 \end{proof}
	 
	 Summing over all the pairs $(k,m)$ in \Cref{lemma:bound_on_insertion_of_player}, for $k=1,\dots, K$ and $m\in M$ we immediately have that 
	 
	 \[||\Lambda(S\cup\{i_0\})-\Lambda(S)||_1\le \sum_{m=1}^M \ell_{i_0}(m)=||\ell_{i_0}||_1.\]
	 Note that the right hand side of this inequality depends only on the player $i_0$’s tent function, and not on $N$, not on $K$, and not on
which coalition $S$ the player $i_0$ is being inserted into. Therefore, we immediately have the following corollary.
	 
	 \begin{corollary}
	 \label{corollary:bound_on_insertion_of_player}
	 	Suppose that $f$ is $L_f$-Lipschitz with respect to the $||\cdot||_1$ norm on persistence landscapes vectors in $\mathbb{R}^{K\times |M|}$, i.e., $|f(x)-f(y)|\le L_f ||x-y||_1$. Then, for every permutation $\pi$ and every player $i$,
		\[|v_f(P_i^{\pi}\cup\{i\})-v_f(P_i^{\pi})|\le L_f||\ell_i||_1,\]
		independent of $N$, of $\pi$, and of where in the permutation $\pi$ the player $i$ falls.
	 \end{corollary}
	 
	 We now recall the complexity of estimating Shapley values via Monte Carlo sampling.
	 	 
	 \begin{theorem}[Hoeffding's inequality]{\cite[Theorem 2]{Hoeffding1963}}
	 \label{theorem:hoeffding_inequality}
	 	Let $X_1,\dots,X_n$ be independent random variables, such that $a_j\le X_j\le b_j$ almost surely for all $1\le j\le n$, and let $S_n= \sum_{j=1}^nX_j$. Then for every $\delta>0$,
		
		\[\text{Pr}\left(\left| S_n-\mathbf{E}[S_n] \right|\ge \delta \right)\le 2\text{exp}\left(-\dfrac{2\delta^2}{\sum_{j=1}^n(b_j-a_j)^2}\right).\]
	 \end{theorem}
	 
	 Suppose the $X_j$'s are i.i.d. copies of a single random variable $X$ confined to a common interval, $a\le X\le b$ almost surely, of length $L=b-a$. Note that $\sum_{j=1}^n L^2=nL^2$ and if $\mu=\mathbf{E}[X]$ then $\mathbf{E}[S_n]=n\mu$. Furthermore, the event $|S_n-\mathbf{E}[S_n]|\ge \delta$ is identical to the event $|\frac{1}{n}S_n-\mu|\ge \frac{\delta}{n}$. Letting $\delta=n\varepsilon$, by \Cref{theorem:hoeffding_inequality} we immediately have the following:
	
	\begin{corollary}
	\label{corollary:hoeffdings_inequality_iid}
		Let $X_1,\dots,X_n$ be i.i.d. copies of a single random variable $X$ confined to a common interval $a\le X\le b$, such that $a\le X_j\le b$ almost surely for all $1\le j\le n$, and let $S_n= \sum_{j=1}^nX_j$, and $L=b-a$. Then for every $\varepsilon>0$,
		
		\[\text{Pr}\left(\left| \dfrac{1}{n}S_n-\mu \right|\ge \varepsilon \right)\le 2\text{exp}\left(-\dfrac{2n\varepsilon^2}{L^2}\right).\]
	\end{corollary}
	
	This corollary will be used to give us the error bounds on approximating Shapley values for nonlinear models. Estimation of Shapley values has a long history. Castro, G\'omez, and
Tejada in \cite{Castro2009} analyze Monte Carlo estimation of the Shapley value and give a confidence interval based on the Central Limit Theorem. This bound for the error in the approximation is only asymptotic, as the number of samples grows to infinity, and offers no finite sample guarantee. Maleki et al. \cite{maleki2014boundingestimationerrorsamplingbased} address this gap, providing non-asymptotic bounds on the approximation error. In particular they provide a Chebyshev-type bound when the variance is known, and a Hoeffding-type bound of essentially the same form as \Cref{theorem:shapley_values_error_bound} below when only the range is known. However, for a general game $v$, that range is bounded crudely (e.g. by the total spread of $v$ over all coalitions) and potentially growing with the player size $N$. \v{S}trumbelj and Kononenko in \cite{trumbelj2013} popularized permutation sampling as a practical, model-agnostic estimator for feature attribution more broadly, again without exploiting any structure specific to $v$.  Our bound in \Cref{theorem:shapley_values_error_bound} has no dependence on $N$.

	  In particular, for a tent function of player $i$, let $R_i\od L_f\cdot ||\ell_i||_1$ where $L_f$ is the Lipschitz constant of the model $f$. Then we have the following.
	  \begin{theorem}
	  \label{theorem:shapley_values_error_bound}
	 	Fix a player $i$, and let $\hat{\varphi}_i(v_f)$ be the \Cref{alg:nonlinear_model} estimate after $n$ i.i.d. sampled permutations. For any $\varepsilon>0$, $\delta \in (0,1)$,
		\[n\ge \dfrac{2R_i^2}{\varepsilon^2}\log \dfrac{2}{\delta}\Longrightarrow \text{Pr}(|\hat{\varphi}_i(v_f)-\varphi_i(v_f)|\ge \varepsilon)\le\delta,\]
		with $R_i$ as in \Cref{corollary:bound_on_insertion_of_player}. A guarantee holding simultaneously for all $N$ players, via a probability union bound, using a per player confidence level of $\frac{\delta}{N}$ only increases this to $n= O(R^2_{\max}\varepsilon^{-2} \log(\dfrac{N}{\delta}))$ with $R_{\max} =\max_iR_i$ which is logarithmic and not polynomial, in $N$.
	 \end{theorem}
	
	\begin{proof}
	The per-permutation summand for player $i$ in \Cref{alg:nonlinear_model} is exactly the marginal contribution $v_f(P_i^{\pi}\cup \{i\})-v_f(P_i^{\pi})$. This is an i.i.d. sample, generated by the random draw of $\pi$, of a random variable with mean $\varphi_i(v_f)$. This random variable, by \Cref{corollary:bound_on_insertion_of_player}, takes values in the interval $[-R_i,R_i]$. Applying \Cref{corollary:hoeffdings_inequality_iid}, with $L = 2R_i$, yields 
	\[\text{Pr}\left(|\hat{\varphi}_i(v_f)-\varphi_i(v_f)|\ge \varepsilon\right)\le 2\exp\left(\dfrac{-2n\varepsilon^2}{(2R_i)^2}\right)=2\exp\left(\dfrac{-n\varepsilon^2}{2R_i^2}\right).\]
Requiring that $2\exp\left(\dfrac{-n\varepsilon^2}{2R_i^2}\right)\le \delta$ yields $n\ge \dfrac{2R_i^2}{\varepsilon^2}\log \dfrac{2}{\delta}$.
	\end{proof}
	
	\begin{remark}
	The $\varepsilon^{-2}$ dependence in \Cref{theorem:shapley_values_error_bound} is the ordinary, unavoidable rate of any Monte Carlo mean estimator. What is special about the landscape featurization is that the constant multiplying $\varepsilon^{-2}$, namely $R^2_i$, does not increase with the number of players $N$. A priori, we might expect the $2^N$ coalition space for the game $v_f$ to make the marginal contributions more erratic, and hence require more samples for larger $N$. However,  \Cref{lemma:bound_on_insertion_of_player} shows that this is not the case.  In particular, a larger pool of coalitions cannot make one point’s own insertion effect any larger than its own persistence already allows. 
	\end{remark}
	
	\subsection{Tree ensembles}
	
	\label{section:tree_ensembles}

	Tree ensembles (e.g. decision trees, random forests and XGBoost) are an exception, where it is possible to compute the Shapley values of features exactly and not just estimates in polynomial time using dynamic programming \cite{Lundberg2020}. Therefore, we can compute exact Shapley values for a tree ensemble trained on a featurization of persistence diagrams. However, how to combine these values to give the attribution to points in a persistence diagram is currently open. 
	
	\subsection{Stability}
	
	Here we show Stability for Shapley values of the game $v_f$. Let $D=\{(b_i,d_i)\}_{i=1}^N,D'=\{(b_i,d_i)\}_{i=1}^N$ be two persistence diagrams indexed by the same set $\{1,\dots ,N\}$ via a fixed matching between them, with diagonal points added as in \Cref{section:stability_landscape_game}, so that point $i$ of $D$ and point $i$ of $D'$ are the two ends of one matched pair. Let 
	\[\varepsilon\od \max_i \max(|b_i-b_i'|,|d_i-d_i'|),\]
	so that $\varepsilon=d_B(D,D')$ when the matching is optimal. Let $\ell_i(m)$ and $\ell'_i(m)$ be the tent functions of the point $i$ in $D$ and $D'$, respectively.  Let $\varphi_i(v_f;D)$ and $\varphi_i(v_f;D')$ be the Shapley values of a player $i$, for the games $v_f$ constructed from the diagrams $D$ and $D'$, respectively. Recall that 
	\[|\ell_i(m)-\ell_i'(m)|\le \varepsilon, \forall i,m.\]
	
	Throughout the following, $M$ is a finite subset of $\mathbb{R}$, the points at which we discretize the landscape functions. Let $\Lambda(D(S))\in \mathbb{R}^{K\times |M|}$ be the landscape constructed from the coalition $S$, and let $||\ell_i(m)||_1\od \sum_{m\in M}\ell_i(m)$ be the $1$-norm of the $i$-th tent function discretized at the points in $M$.
	For $m\in M$ we define $L(m)\od\sum_{i=1}^N|\ell_i(m)-\ell_i'(m)|$, and let
	\[B(K,\varepsilon)\od \sum_{m\in M}\min(L(m),K\varepsilon).\]
	
	For $x\in \mathbb{R}^n$, let $\hat{x}$ be its nonincreasing rearrangement. This sorting map is $1$-Lipschitz with respect to the $\ell^1$ norm, a consequence of the classical Hardy-Littlewood-P\'olya rearrangement inequality \cite[Theorem 368]{ECT1953}. In particular we have the following.
		
	\begin{lemma}
	\label{lemma:sorting_is_1_lipschitz}
	For all $x,y\in \mathbb{R}^n$, $\sum_{j=1}^n |\hat{x}_j-\hat{y}_j|\le \sum_{j=1}^n |x_j-y_j|$.
	\end{lemma}
		
	We use this fact to find bounds on the $1$-norm of the persistence landscapes constructed from a coalition $S$ from diagrams $D$ and $D'$.

	\begin{lemma}
	\label{lemma:coalition_bounds_between_matchings}
	For every coalition $S \subset \{1,\dots,N\}$, every $m \in M$ and every $k\le K$,
		\[|v_{(k,m)}(S;D)- v_{(k,m)}(S;D')|\le \varepsilon,\]
		and
		\[||\Lambda(D(S))- \Lambda(D'(S))||_1\le B(K,\varepsilon).\]
	\end{lemma}
	
	\begin{proof}
		Let $S$ be a coalition and let $m\in M$. Let $x \in \mathbb{R}^{|S|+K}$ consist of the values $\ell_i(m)$, $i \in S$, followed by $K$ zeros.  Similarly, let $y$ be defined in the same way from $\ell_i'(m)$, using the same ordering of $S$. Let $\hat{x}$ and $\hat{y}$ be nonincreasing rearrangements of $x$ and $y$, respectively. Tent values are nonnegative, so the appended zeros sort to the end, and the first $K$ entries of $\hat{x}$ are precisely $v_{(1,m)}(S; D),\dots ,v_{(K,m)}(S; D)$. Note that by construction $v_{(k,m)}(S; D) = 0$ when $|S|<k$. The same is true for $y$ and $D'$. Padding with $0$s changes neither  $\sum_j|x_j-y_j|$ nor the coordinatewise bound $|x_j-y_j|\le \varepsilon$, for all $j$, as shown in the proof of \Cref{theorem:stability_for_landscape_games}.
		
		Additionally, the rank argument in the proof of \Cref{theorem:stability_for_landscape_games} applies verbatim to any two vectors on a common index set whose entries differ by at most $\varepsilon$. Therefore 
		\[|\hat{x}_k-\hat{y}_k|=|v_{(k,m)}(S;D)-v_{(k,m)}(S;D')|\le \varepsilon,\] 
		for all $k$. Therefore $\sum_{k=1}^K|\hat{x}_k-\hat{y}_k|\le K\varepsilon$. Truncating the sum in \Cref{lemma:sorting_is_1_lipschitz} to the first $K$ terms yields
		\[\sum_{k=1}^K |\hat{x}_j-\hat{y}_j|\le \sum_{j} |x_j-y_j|=\sum_{i\in S}|\ell_i(m)-\ell_i'(m)|\le L(m).\]
		Finally, summing over $m\in M$ gives 
		\[||\Lambda(D(S))- \Lambda(D'(S))||_1\le B(K,\varepsilon). \qedhere\]
	\end{proof}
		
	\begin{theorem}
	\label{theorem:stability_nonlinear_model}
	Let $D,D'$ be persistence diagrams as above, and let $f$ be $L_f$-Lipschitz with respect to the $1$-norm on the flattened $K\times |M|$ landscape vector. Then, for every player $i$ we have 
	\[|\varphi_i(v_f;D)-\varphi_i(v_f;D')|\le \min\left(2L_fB(K,\varepsilon), L_f(||\ell_i||_1+||\ell_i'||_1)\right) \le 2L_fK|M|\varepsilon.\]
	\end{theorem}
	
	\begin{proof}
	By \Cref{lemma:coalition_bounds_between_matchings} and because $f$ is Lipschitz, we have
	\[\left |v_f(S;D)-v_f(S;D')\right|=\left|f(\Lambda(D(S)))-f(\Lambda(D'(S)))\right|\le L_fB(K,\varepsilon),\]
	for every coalition $S$. Then, by above inequality and the triangle inequality we have
\begin{align*}
\left | (v_f(S\cup \{i\};D)-v_f(S;D))-(v_f(S\cup \{i\};D')-v_f(S;D')))\right|\le\\
 \left |v_f(S\cup\{i\};D)-v_f(S\cup\{i\};D')\right| +\left |v_f(S;D)-v_f(S;D')\right|& \le 2L_fB(K,\varepsilon),
 \end{align*}
	for every coalition $S\subset \{1,\dots, N\}\setminus \{i\}$. The Shapley value is the average of all such marginal contributions $v_f(S\cup \{i\};D)-v_f(S;D)$. Averaging the marginal contributions over all possible coalitions $S$ we thus have
 \[ |\varphi_i(v_f;D)-\varphi_i(v_f;D')|\le 2L_fB(K,\varepsilon).\]
 
 	Furthermore, by \Cref{corollary:bound_on_insertion_of_player} we have the inequality $|v_f(S\cup\{i\};D)-v_f(S;D)|\le L_f||\ell_i||_1$, for every coalition $S$. Therefore averaging these marginal contributions over all coalitions gives $|\varphi_i(v_f;D)|\le L_f||\ell_i||_1$. Similarly we have $|\varphi_i(v_f;D')|\le L_f||\ell_i'||_1$ and by triangle inequality we have
	 \[ |\varphi_i(v_f;D)-\varphi_i(v_f;D')|\le L_f(||\ell_i||_1+||\ell_i'||_1).\]
	\end{proof}
	
	From this result we immediately have the following.
	
	\begin{corollary}
	\label{corollary:landscape_game_stability}
		For every matched player $i$ and every landscape coordinate $(k,m)$,
		\[\left|\varphi_i(v_{(k,m)};D)-\varphi_i(v_{(k,m)};D')\right|\le \min(2\varepsilon,\max(\ell_i(m),\ell_i'(m))).\]
	\end{corollary}
	
	\begin{proof}
		By \Cref{lemma:coalition_bounds_between_matchings} and the triangle inequality, the marginal contributions of player $i$ to any coalition $S$ in the games $v_{(k,m)}(\cdot ; D)$ and $v_{(k,m)}(\cdot ; D')$ differ by at most $2\varepsilon$. Thus, averaging gives 
		\[\left|\varphi_i(v_{(k,m)};D)-\varphi_i(v_{(k,m)};D')\right|\le 2\varepsilon.\]
		
		Additionally, in the proof of \Cref{theorem:closed_form_landscapes} we derived 
		\[\varphi_i(v_{(k,m)};D)=\int_0^{\ell_i(m)}\dfrac{\mathbf{1}_{\{B(\tau)\ge k\}}}{B(\tau)}d\tau\in [0,\ell_i(m)].\]
		Similarly we also have $\varphi_i(v_{(k,m)};D')\in [0,\ell_i'(m)]$. The difference of two numbers in $[0,\ell_i(m)]$ and $[0,\ell_i'(m)]$ is at most $\max(\ell_i(m),\ell_i'(m))$ in absolute value and this gives 
		\[\left|\varphi_i(v_{(k,m)};D)-\varphi_i(v_{(k,m)};D')\right|\le \max(\ell_i(m),\ell_i'(m)).\qedhere\]
	\end{proof}
	
	\begin{remark}
		\Cref{theorem:stability_for_landscape_games} compares the players at rank $r$ in $D$ and in $D'$. Under a perturbation, the matched player $i$ may change rank, so \Cref{theorem:stability_for_landscape_games} does not by itself show that $\varphi_i(v_{(k,m)})$ is stable for a fixed matched player. \Cref{corollary:landscape_game_stability} provides this per player statement.
	\end{remark}
	
	\section{Experiments}
	
	\label{section:experiments}
	Here we show LandscapeSHAP can explain how models are making decisions based on persistent homology classes. We showcase the power of this method with both a regression and a classification task. All the hyperparameter choices are detailed in \Cref{section:hyperparameters}.

	\subsection{Curvature regression}
		Here we investigate how our method can be applied in practice for a linear regressor.  In \cite{Bubenik2020}, the authors sample randomly $1000$ points on a disc with radius $1$ of constant Gaussian curvature $\kappa$ and predict $\kappa$ from the Vietoris-Rips persistence diagrams of the resulting point-cloud with respect to the intrinsic metric. They give mathematical evidence that small bars in these persistence diagrams should be good predictors of the curvature, and are able to train different machine learning models with good $R^2$ scores, such as a Support Vector Regression (SVR). 
		
	\subsubsection{Average landscapes detect curvature}
		We reproduced the dataset of \cite{Bubenik2020}. In particular, for $\kappa\in [-2,2]$, starting at $-2$ and taking increments of $0.04$ towards $2$, we sample $1000$ points from the disc of radius $1$ with constant Gaussian curvature $\kappa$ (\Cref{fig:curvature_examples}). 
		
	\begin{figure}[H]
	\centering
		\includegraphics[scale=0.45]{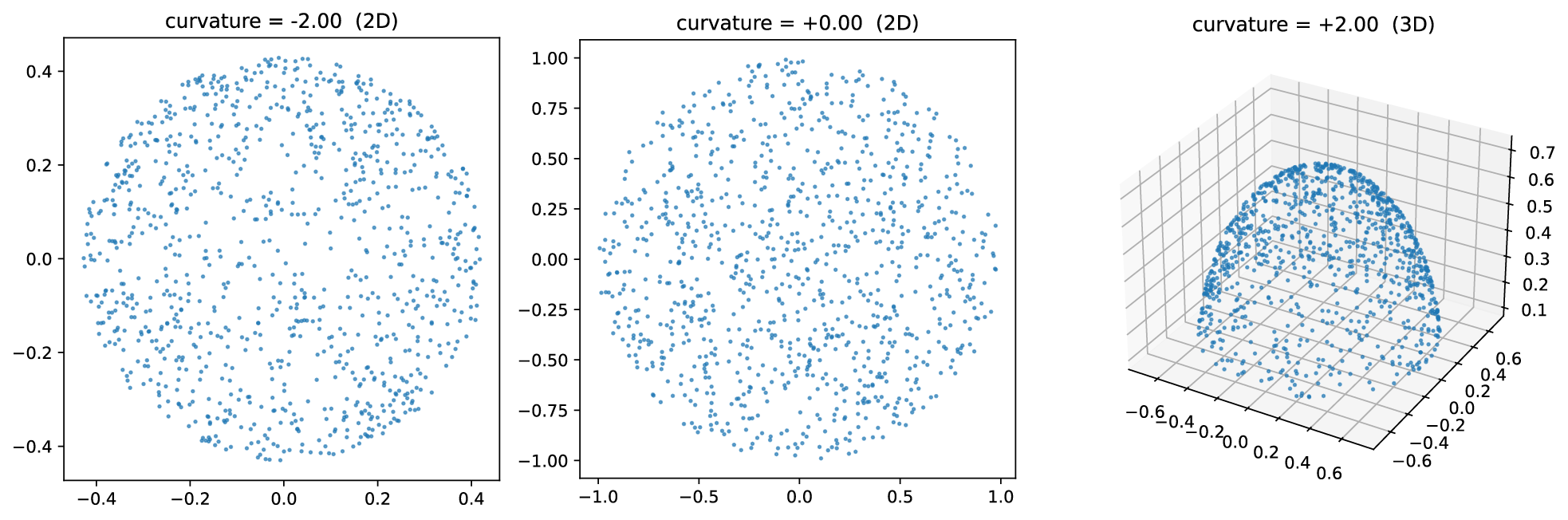}
		\caption{Samples from discs of radius $1$ of constant curvatures. Left, $1000$ points sampled from disc of curvature $-2$, middle $1000$ points sampled from curvature $0$, right $1000$ points sampled from curvature $2$.}
		\label{fig:curvature_examples}
	\end{figure}
	
	We do this $100$ times for each curvature value $\kappa$. Then we compute the Vietoris-Rips filtration persistence diagrams in dimension $1$ for each point cloud using the Python library \cite{ctralie2018ripser} of the Ripser algorithm \cite{Bauer2021Ripser}. For each $\kappa$, we compute the average landscapes, resulting in $101$ vectors in a $2500$-dimensional space. The first principal component of this data is almost perfectly correlated ($r=-0.9986$) with the curvature values $\kappa$ (\Cref{fig:curvature_pca}).

	\begin{figure}[H]
	\centering
		\includegraphics[scale=0.35]{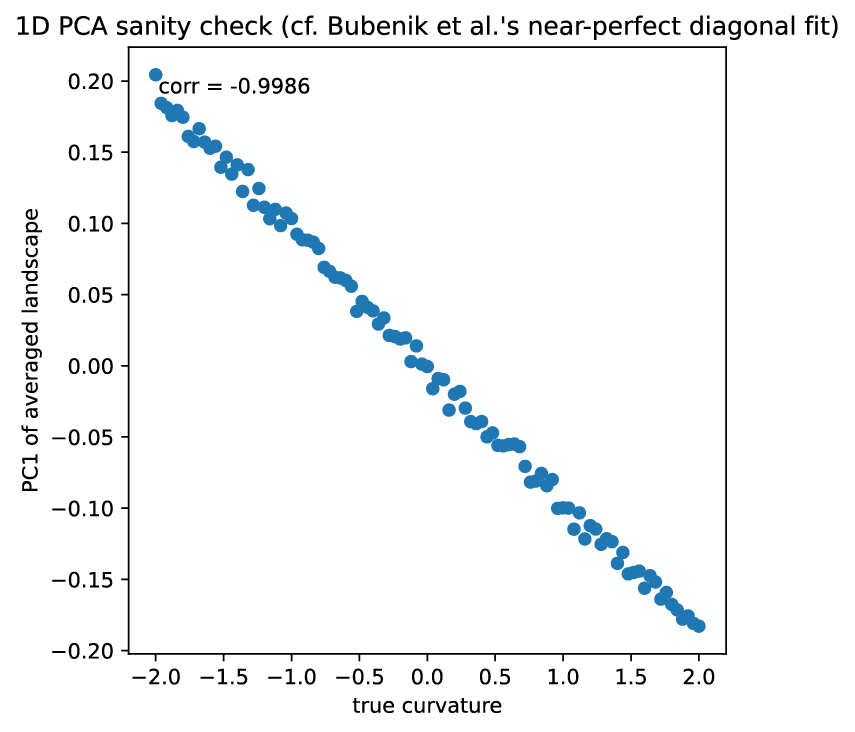}
		\caption{The first principal component ($y$-axis) plotted against the true underlying curvature values associated with the average persistence landscape vectors.}
		\label{fig:curvature_pca}
	\end{figure}
	
		\subsubsection{Explaining a linear SVR}
	The fact that persistence landscapes of this data set are organized continuously with respect to curvature or that their first principal component can be used to predict curvature is not new \cite{Bubenik2020}. Therefore, a linear model seems as good as any for a regression task which predicts curvature from the average landscapes. We trained a linear SVR model to do precisely that. We now illustrate how LandscapeSHAP explains this model's predictions. We computed the Shapley values for the averaged persistence landscapes of the SVR just as described in \Cref{section:average_landscapes}. The results are shown in \Cref{fig:shap_curvature}. Based on the results displayed in \Cref{fig:shap_curvature} we make the following observations.
	
	\begin{itemize}[left=0pt]
		\item The sum of the Shapley values equals the model's prediction minus the empty-diagram baseline $f(\Lambda(\varnothing))$, scaled by $1/100$ (this is the efficiency axiom being satisfied).
		\item The SVR model is allocating (blue) negative Shapley values to persistence classes that appear in point clouds of negative curvature. Summing the Shapley values of these persistent homology classes pushes the model towards predicting negative curvature. On the other hand, the model is allocating (red) positive Shapley values to persistence classes that appear in point clouds of positive curvature. 
		\item For the curvature $\kappa=0$ case, the sum of positive and negative Shapley values cancels out giving a prediction close to $0$ to the model.
		\item The model assigns negative Shapley values to persistence classes that have larger death values (higher on the $y$-axis) and positive Shapley values to the persistence classes that have smaller death values (lower on the $y$-axis). This is in agreement with hyperbolic distances on average being larger than spherical distances.
		\item The most relevant (highest absolute Shapley values) persistent homology classes sometimes correspond to shorter bars in the persistence diagram. This agrees with the mathematical theory derived in \cite{Bubenik2020}. 
	\end{itemize}
	 Therefore, LandscapeSHAP confirms that the model is making its predictions in a reasonable way based on the observations above.
	
	\begin{figure}[H]
	\centering
		\includegraphics[scale=0.3]{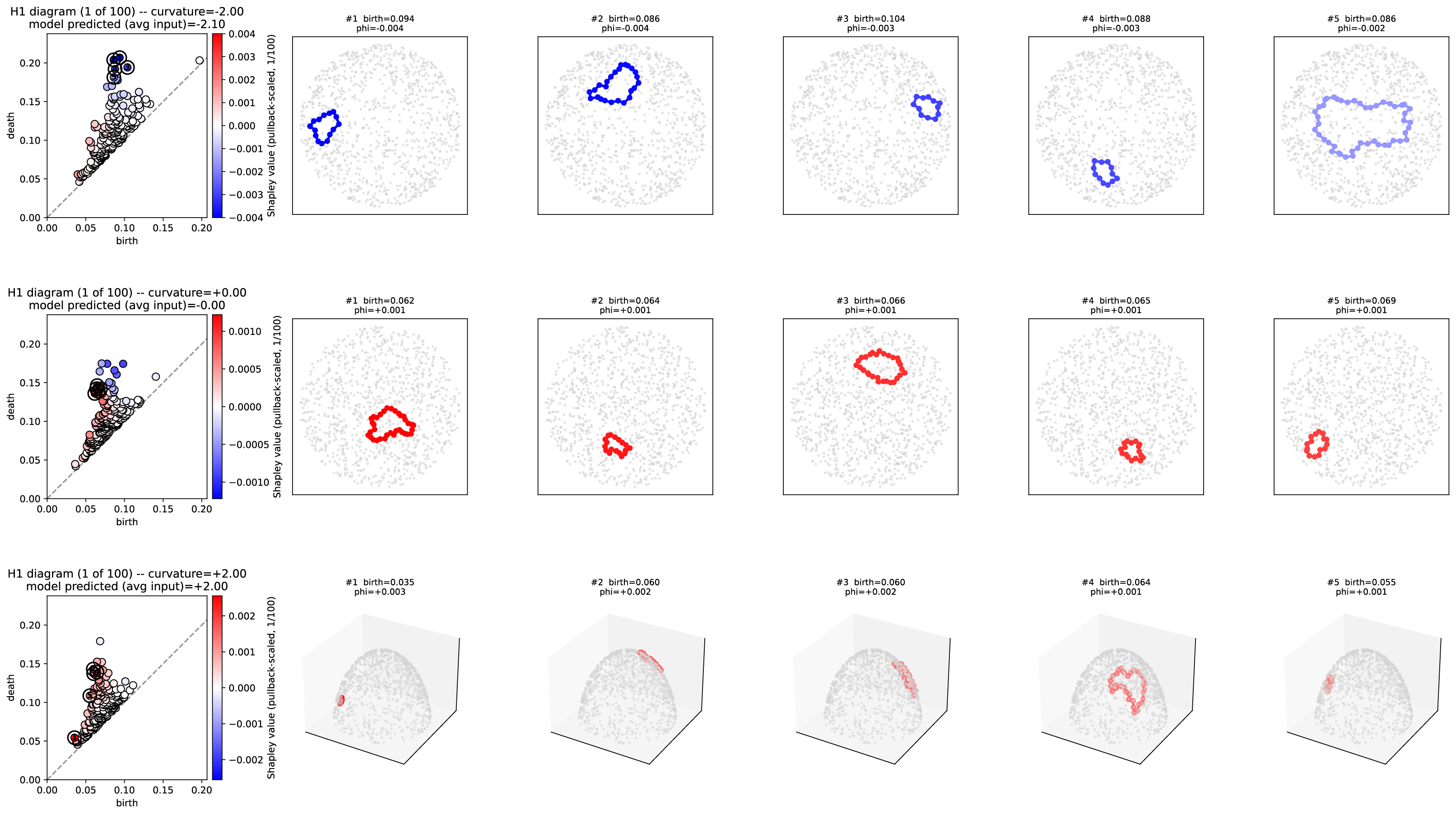}
		\caption{LandscapeSHAP explaining (linear) curvature regression. Top, a persistence diagram coming from a point cloud sampled from curvature $\kappa=-2$, colored by using the Shapley values of the points. The top five Shapley values (in absolute value) are selected and the corresponding representative cycles are shown. Middle, a persistence diagram coming from a point cloud sampled from curvature $\kappa=0$, colored by using the Shapley values of the points. The top five Shapley values (in absolute value) are selected and the corresponding representative cycles are shown. Bottom, a persistence diagram coming from a point cloud sampled from curvature $\kappa=2$, colored by using the Shapley values of the points. The top five Shapley values (in absolute value) are selected and the corresponding representative cycles are shown.}
		\label{fig:shap_curvature}
	\end{figure}
	
	\subsection{Dynamical system classifier}
	Here we investigate how LandscapeSHAP can be applied in practice for a (linear) classifier.

	\begin{figure}[H]
	\centering
		\includegraphics[scale=0.35]{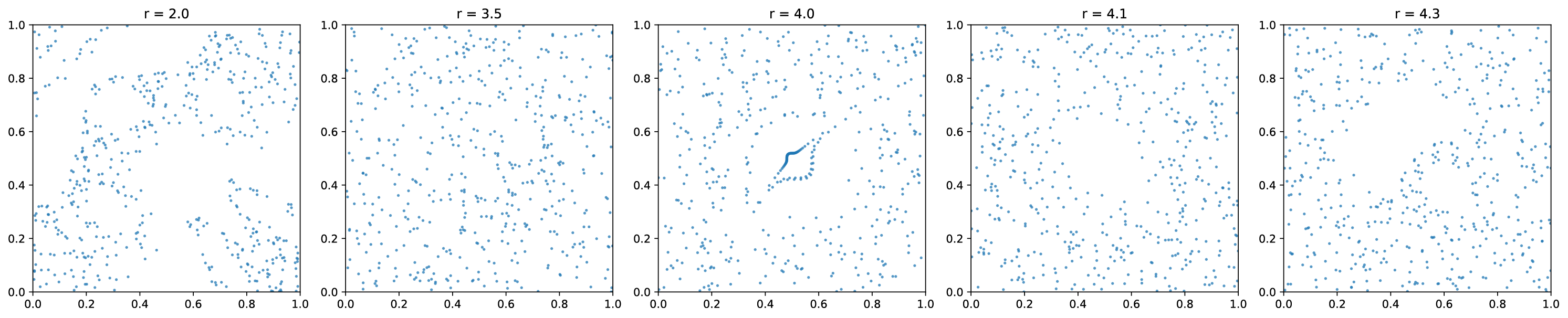}
		\caption{Point clouds with $1000$ points each of the dynamical system for $r\in \{2,3.5,4,4.1,4.3\}$ going from left to right.}
		\label{fig:dynamical_system}
	\end{figure}
	
	We consider a one-parameter family of discrete dynamical systems on a flat torus, defined by a linked twist map where each twist is governed by the logistic map \cite{Hertzsch2007}:

\begin{align*}
x_{n+1}&=x_n+r\cdot y_n(1-y_n)\mod 1\\
y_{n+1}&=y_n+r\cdot x_{n+1}(1-x_{n+1})\mod 1
\end{align*}

For each $r\in \{2,3.5,4,4.1, 4.3\}$, a point cloud with $1000$ points is generated with
an initial condition drawn uniformly at random from the unit square $[0, 1]^2$ (the first 500 iterates are discarded  and the next 1000 iterates form the point cloud). An example of five point clouds, one for each $r$ value, is shown in \Cref{fig:dynamical_system}. This process is repeated $100$ times, resulting in $500$ point clouds. We compute the alpha complex persistent homology of each point cloud in dimension $1$ using the GUDHI library \cite{gudhi:AlphaComplex}. To keep things simple we use the Euclidean metric in the plane, and not a torus metric. 
	
	We train a linear Support Vector Machine (SVM) classifier and a radial kernel SVM classifier on the resulting persistence landscapes to predict the parameter $r$. Once the model has been trained we use LandscapeSHAP to explain the model's prediction, on the same 5 point clouds from \Cref{fig:dynamical_system}.

	\subsubsection{Linear model}
	Here we examine the Shapley values of the linear SVM model and use them to explain how the model is making predictions. The first column in \Cref{fig:linear_shap_dynamical_system} shows the resulting persistence diagrams, where each point is colored according to its Shapley value coming from the linear classifier. Based on the results displayed in \Cref{fig:linear_shap_dynamical_system} we make the following observations.
	
	\begin{itemize}[left=0pt]
		\item The two highest $r$ value rows ($r=4.1,4.3$) are for point clouds that have one or two large loops with large Shapley values  that dominate everything else in the same diagram. Thus, a single big topological hole in the dynamics is essentially deciding the classification. The larger loops are positive in Shapley value and are pushing the model's prediction towards higher $r$ values.
		\item  The $r=2,3.5$ rows have no dominant loop. The largest five Shapley values are an order of magnitude smaller in absolute value than in the $r=4.1,4.3$ cases and mixed in sign. The representative cycles are also smaller, scattered loops. Interestingly, the point with the largest persistence in the diagram has Shapley of close to $0$. Usually the mantra in persistence theory is that the points with larger persistence are more important. However, LandscapeSHAP allows us to determine if the model thinks they are important, and in this case it seems they are not. 
		\item The $r=4$ row is an interesting case. The model misclassified it to $r=3.5$ and its top Shapley values are noticeably weaker than the unambiguous $r=4.1,4.3$. The small Shapley values are indicative that the model is genuinely unsure about its prediction here.
	\end{itemize}
	Therefore, LandscapeSHAP is offering valuable insight into how a linear SVM is making its predictions based on the observations above. In particular, the misclassification for the  $r=4$ example is not an accident but likely a product of the model's uncertainty. Perhaps a linear model is not the best choice for this type of data (at least for correctly predicting the values of $r=4$), or perhaps the number of layers we used to construct the persistence landscapes was not the right hyperparameter. Whatever the case may be, LandscapeSHAP can be used on a different choice of a model (by approximating Shapley values for nonlinear models), or landscapes with a different number of layers and compare the results. Additionally, we only used persistent homology features in dimension $1$. Adding features from persistent homology in dimension $0$ could also improve the model's performance as was the case in \cite{JMLR:v18:16-337}. Our goal here is not to construct the ``best model" to predict the dynamics parameter $r$ for this dataset, but simply to showcase how LandscapeSHAP can be used to explain a given model's prediction. In real applications, LandscapeSHAP offers an additional way for practitioners to examine a model's validity, compare different classes of models or optimize for hyperparameters.
		
	\begin{figure}[H]
	\centering
		\includegraphics[scale=0.27]{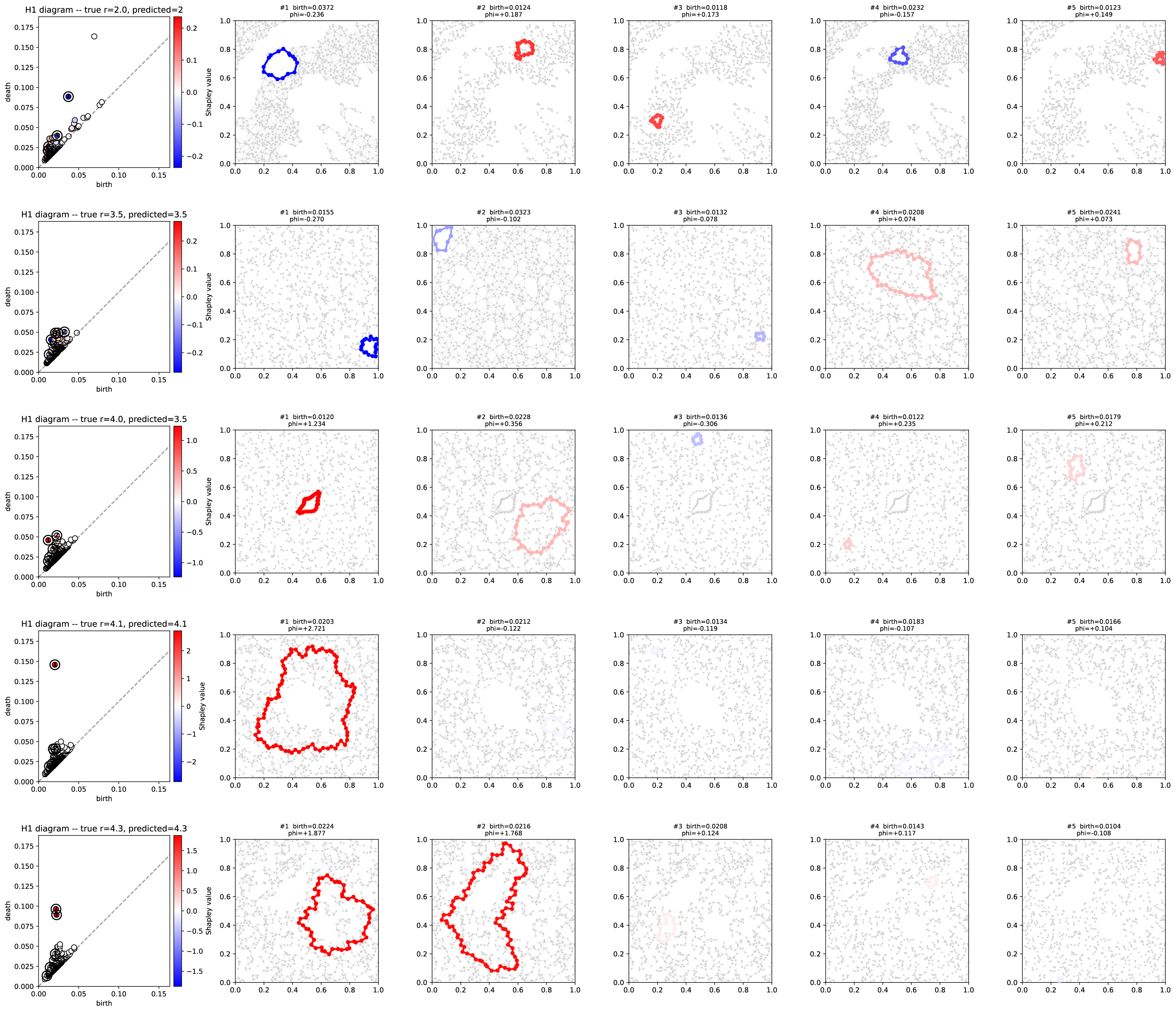}
		\caption{LandscapeSHAP explaining (linear) dynamical system classification. First row, a persistence diagram coming from a point cloud sampled from dynamics with $r=2$, colored by using the Shapley values of the points. Second row, a persistence diagram coming from a point cloud sampled from dynamics with $r=3.5$, colored by using the Shapley values of the points. Third row, a persistence diagram coming from a point cloud sampled from dynamics with $r=4$, colored by using the Shapley values of the points. Fourth row, a persistence diagram coming from a point cloud sampled from dynamics with $r=4.1$, colored by using the Shapley values of the points. Fifth row, a persistence diagram coming from a point cloud sampled from dynamics with $r=4.3$, colored by using the Shapley values of the points. The top five Shapley values (in absolute value) are selected and the corresponding representative cycles are shown, for each row.}
		\label{fig:linear_shap_dynamical_system}
	\end{figure}

	\subsubsection{Linear model approximation}
	\label{section:mc-linear}

	We took the same linear SVM classifier from before and approximated the Shapley values of points on the persistence diagrams by using \Cref{alg:nonlinear_model}. We wish to investigate how many permutations are needed for this model to approximate the Shapley values correctly. We will then use this number of permutations to approximate Shapley values for a nonlinear model. For five point clouds with different $r$ values, we see in \Cref{fig:shapley_values_errors_linear_model} that $n=3000$ permutations are good enough. 	
	
	\begin{figure}[H]
	\centering
		\includegraphics[scale=0.32]{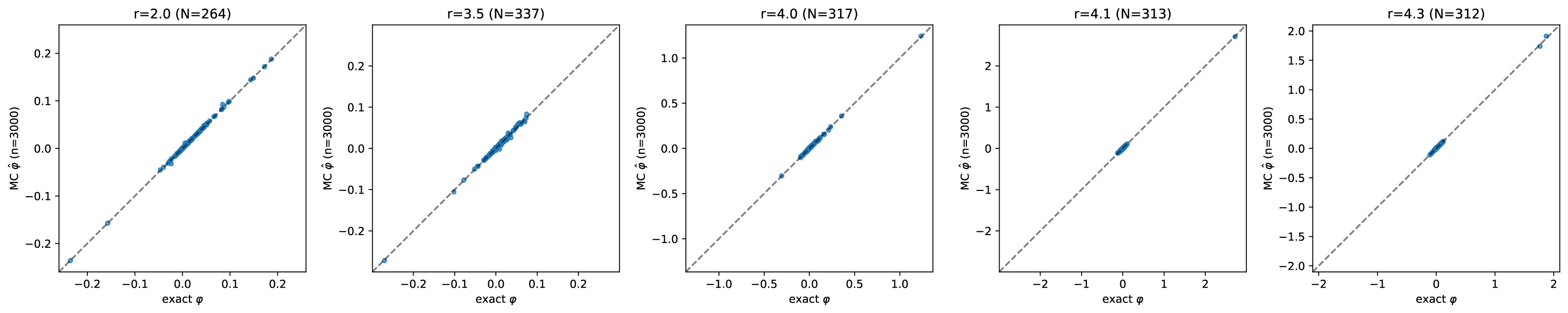}
		\caption{Estimated Shapley values vs exact Shapley values of the linear SVM model for the point clouds in \Cref{fig:linear_shap_dynamical_system}.}
		\label{fig:shapley_values_errors_linear_model}
	\end{figure}
		
	Note that the persistence diagrams in all cases in \Cref{fig:shapley_values_errors_linear_model} have more than $N=250$ points each, so the coalition space is at least of size $2^{250}$. The fact that only $3000$ permutations are sufficient to approximate Shapley values to this level of accuracy is encouraging.
	
	For the same five example point clouds, we run $8$ independent Monte Carlo (\Cref{alg:nonlinear_model}) estimators, with different random seeds, up to $n$ permutations, where $n\in \{10, 30, 100, 300, 1000, 3000\}$. Against the exact values we track the mean absolute error, normalized by the average absolute Shapley value for that diagram. Thus, diagrams with different scales are comparable. We also compute the Pearson correlation between the Monte Carlo estimate and the exact values to check if the estimate at least gets the relative ordering/sign of Shapley values right. The results are shown in \Cref{fig:error_estimations_top_5}.
		
	\begin{figure}[H]
	\centering
		\includegraphics[scale=0.5]{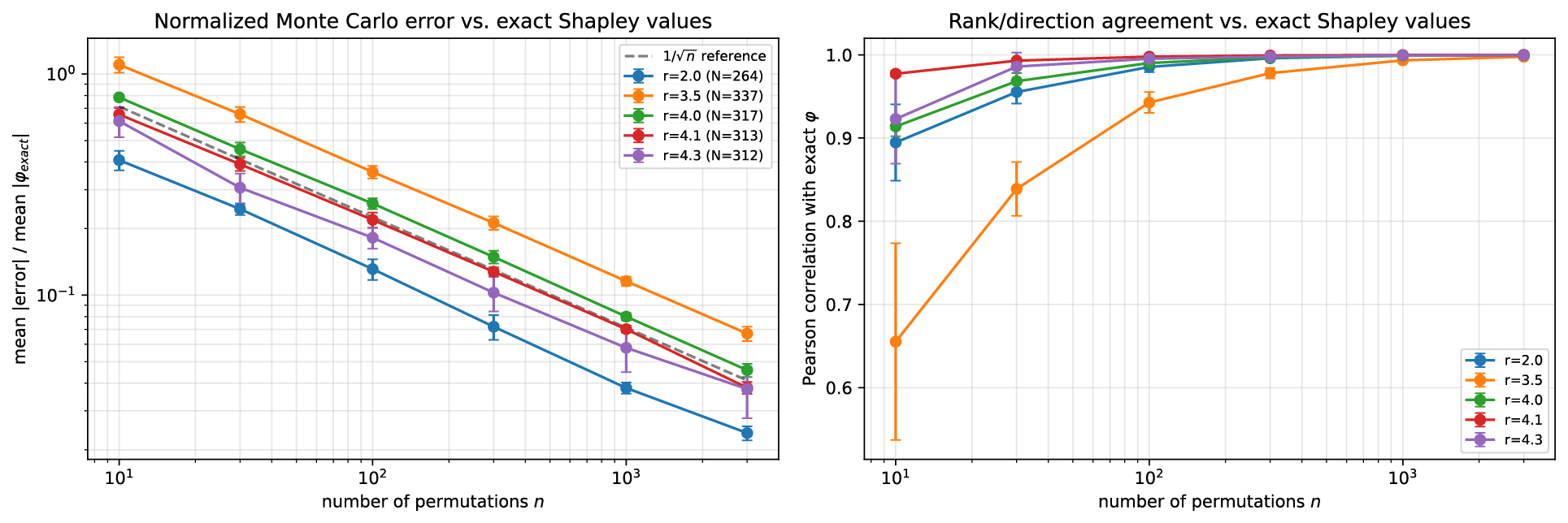}
		\caption{Convergence analysis of estimated Shapley values for a linear SVM. Left, the plot of mean error normalized by mean absolute Shapley values for five example diagrams against the number of permutations used on the $\log$ scale. Right, the Pearson correlation between estimates and exact Shapley values (across 8 different random seed estimations) plotted against the number of permutations used on the $\log$ scale.}
		\label{fig:error_estimations_top_5}
	\end{figure}
	
Across all five example diagrams, the Monte Carlo estimate error decays at the  $O(\frac{1}{\sqrt{n}})$, the rate \Cref{theorem:shapley_values_error_bound} predicts. The error curves run parallel to the reference line $\frac{1}{\sqrt{n}}$, with their vertical offset varying by diagram. This reflects the range of each diagram's marginal contributions which depends on the Lipschitz constant of the model $L_f$, number of layers $K$, and number of grid points $|M|$ as in \Cref{theorem:shapley_values_error_bound}. Notably, rank/direction agreement with the exact values saturates faster than magnitude does, so far fewer permutations are needed to trust which points matter most than to trust their precise attributed values.
	
	\subsubsection{Nonlinear model approximation}
	Now we train a nonlinear (radial kernel) SVM classifier on persistence landscapes to predict the parameter $r$. We use \Cref{alg:nonlinear_model} to approximate the Shapley values of persistence diagram points for this model. We use $n=3000$ permutation samples to approximate the Shapley values, as this was a good enough sample to approximate the linear model. The results are illustrated in \Cref{fig:nonlinear_shap_dynamical_system}. Based on the results displayed in \Cref{fig:nonlinear_shap_dynamical_system} we make the following observations. 
	
	\begin{itemize}[left=0pt]
		\item For point clouds with $r=4.1$ and $r=4.3$, the model thinks there is one dominant loop ($r=4.1$), or two comparably large ones ($r=4.3$),  that  account for almost the entire prediction. Every other persistent homology class in the top five absolute Shapley values is near zero. Therefore for these point clouds, the nonlinear SVM isn't doing anything qualitatively different from the linear SVM. It's recognizing the same topological features of the data as important, only with different absolute scores. 
		\item For point clouds with $r=2.0$, the linear SVM allocated a negative weight on one specific loop (\Cref{fig:linear_shap_dynamical_system}). On the other hand the nonlinear SVM model, on the same diagram, allocates nothing larger than $\pm 0.03$ to any point. The model therefore isn't leaning on any single topological feature to identify this class. Instead it seems its decision boundary is driven by something more diffuse across the whole corresponding landscape.
		\item For point clouds with $r=3.5$, the linear model's top five absolute Shapley values were small and sign-mixed (\Cref{fig:linear_shap_dynamical_system}). On the other hand nonlinear SVM finds five loops of consistently positive, moderate size. This is a clearer, more distributed positive signal than the linear model extracted from the identical diagram. 
	\end{itemize}
	
	Therefore, LandscapeSHAP is offering valuable insight into how a nonlinear SVM is making its predictions based on the observations above. In particular, the nonlinear kernel isn't uniformly ``more confident" or ``less confident" than the linear model. It allocates credit differently diagram by diagram, sometimes concentrating it more to a few topological features, sometimes spreading it more thinly than the linear model does.	
	
	\begin{figure}[H]
	\centering
		\includegraphics[scale=0.27]{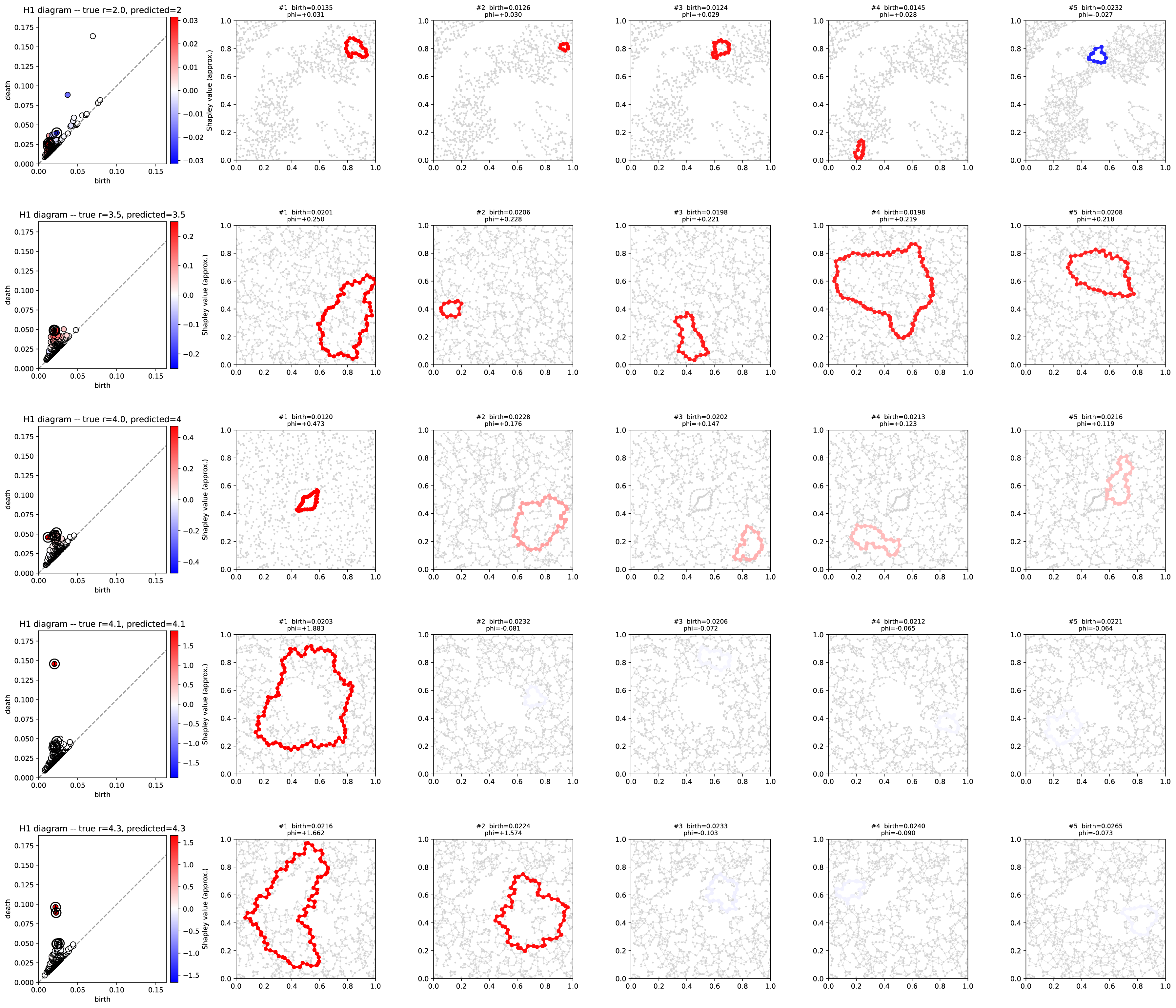}
		\caption{LandscapeSHAP explaining (nonlinear) dynamical system classification. First row, a persistence diagram coming from a point cloud sampled from dynamics with $r=2$, colored by using the Shapley values of the points. Second row, a persistence diagram coming from a point cloud sampled from dynamics with $r=3.5$, colored by using the Shapley values of the points. Third row, a persistence diagram coming from a point cloud sampled from dynamics with $r=4$, colored by using the Shapley values of the points. Fourth row, a persistence diagram coming from a point cloud sampled from dynamics with $r=4.1$, colored by using the Shapley values of the points. Fifth row, a persistence diagram coming from a point cloud sampled from dynamics with $r=4.3$, colored by using the Shapley values of the points. The top five Shapley values (in absolute value) are selected and the corresponding representative cycles are shown, for each row.}
		\label{fig:nonlinear_shap_dynamical_system}
	\end{figure}

	\subsection{Hyperparameters}
	\label{section:hyperparameters}
	
	All random number generators were seeded with $42$. Persistence diagrams are filtered by removing points whose persistence $d_i-b_i$ is below $2\%$ of $G= 1.1\cdot q_{99}$, 	where $q_{99}$ is the $99$th percentile of observed death values. For the dynamical systems, $q_{99}$ is computed over all $500$ diagrams, and for the curvature data over a 		sample of $10$ diagrams ($\kappa\in\{-2,-1,0,1,2\}$, two repeats each). Persistence landscapes are evaluated on a uniform grid $M$ of $|M|=100$ points in $[0,G]$, with $K=25$ 	layers for the curvature regression ($2500$ features) and $K=15$ layers for the dynamical-systems classification ($1500$ features).

	For the curvature regression, each of the $101$ examples is the average landscape of $100$ diagrams (\Cref{section:average_landscapes}). For the classification there 	are $500$ point clouds ($100$ per value of $r$) with $1000$ points each. In both tasks the data is split $75\%/25\%$ into training and test sets (stratified by $r$ for classification), 	and landscape features are standardized to zero mean and unit variance using the training split. The linear models' weights are mapped back to the original landscape coordinates 	by dividing by the scaler's standard deviations, so \Cref{theorem:closed_form_linear_model,theorem:average_landscapes} apply to the coefficients in the original coordinates. For the averaged 	curvature landscapes, the Shapley values shown for a single diagram are the pullback values of \Cref{theorem:average_landscapes} with $\alpha_j=1/100$.

	Models are from scikit-learn: \texttt{LinearSVR} with $C=1$ and $\epsilon=0$ for curvature regression, \texttt{LinearSVC} with $C=1$ and the one-versus-rest scheme for
	the linear classifier, and \texttt{SVC} with an RBF kernel, $C=10$ and \texttt{gamma="scale"} for the nonlinear classifier. The only non-default settings are
	\texttt{max\_iter}$=20000$ for the two linear models and $C=10$ for the RBF kernel. No hyperparameters were tuned. Shapley values of the nonlinear model are estimated with
	\Cref{alg:nonlinear_model} using $n=3000$ permutations per diagram. The convergence study of
	\Cref{section:mc-linear} runs the estimator $8$ times with independent seeds at $n\in\{10,30,100,300,1000,3000\}$ permutations.

	\subsection{Representative cycles}
	Here we explain how the representative cycles in \Cref{fig:shap_curvature,fig:linear_shap_dynamical_system,fig:nonlinear_shap_dynamical_system} are chosen.
	For each of the top five absolute Shapley value points in an example diagram, we recompute the birth edge $(u,v)$ and death triangle via GUDHI's simplex tree \cite{gudhi:FilteredComplexes}, then breadth-first search for a path between $u$ and $v$ using only edges earlier than the birth time. We close the path with the birth edge to obtain a genuine representative cycle. For the dynamical systems point clouds the cycle is drawn over the alpha complex filtration at the birth time of the cycle. For the curvature point clouds the cycle is drawn over the point clouds themselves and the Vietoris-Rips complex filtration is not shown at the birth time.
	
	\subsection{Data availability}
	The code reproducing all experiments is available at \href{https://github.com/milicevic-nikola/LandscapeSHAP.git}{github.com/milicevic-nikola/LandscapeSHAP}.
	
	\section{Additive featurizations}
	\label{section:other_featurizations}
	
	Here we discuss how one can allocate credit based on a model's prediction to the points in a persistence diagram based on additive featurizations with the example of persistence images. A persistence image \cite{JMLR:v18:16-337} is another popular featurization of persistence diagrams. Persistence images are constructed in the following way. Given a persistence diagram $D$, each diagram point $(b_i,d_i)$ is first mapped under the linear transform $T(x,y) \od (x,y-x)$ to the point $T(b_i,d_i) = (b_i, d_i-b_i)$ in birth-persistence coordinates. Let $f: \mathbb{R}^2 \to \mathbb{R}$ be a nonnegative weighting function that vanishes continuously along the $x$-axis and let $\phi_u$ be a differentiable probability distribution centered at each $u\in \mathbb{R}^2$, e.g. a Gaussian. The \emph{persistence surface} is then defined to be
	\[\rho_D(z)\od \sum_{u\in T(D)}f(u)\phi_u(z).\]
	The \emph{persistence image} is the vector of pixel integrals $I(\rho_D)_p\od \int\int_p\rho_D(z)dz$, over a fixed grid of pixels $p$. Restricting the sum defining the persistence surface $\rho_D$ to a coalition $S$ of persistence diagram points gives the characteristic function
	\[I_S(\rho_D)_p\od \sum_{i\in S}f(u_i)\kappa_{(i,p)},\quad \kappa_{(i,p)}\od \int\int_p \phi_{u_i}(z)dz,\]
	where $u_i=T(b_i,d_i)$, and  $\kappa_{(i,p)}$ is the player $i$'s own contribution to pixel $p$ and depends only on player $i$’s own coordinates $u_i$, never on which other points happen to be present.
	 
	 Unlike the persistence landscape games, $v_{(k,m)}$, this characteristic function is additive in $S$, i.e., adding or removing a persistence diagram point does not change any other point's credit/contribution. Therefore there is no genuine credit sharing problem to solve, unlike with the persistence landscape games. For an additive cooperative game $v(S)=\sum_{i\in S} c_i$ its Shapley values for each player are straightforward; $\varphi_i(v)=c_i$ trivially. Composing with a linear model with weights $w_p$ and bias $b$,
	\[h(I(\rho_D))=\sum_{p}w_p I(\rho_D)_p+b,\]
	preserves additivity giving the closed form solution
	\[\varphi_i (v_h)=f(u_i)\sum_p w_p \kappa_{(i,p)},\]
	with $v_h(\varnothing)=0$, i.e.  $h(0)=b$ substracted.
	Therefore, this case is much easier to derive than for an order statistic game like the landscape, where the points compete for a rank and credit allocation needs to be carefully derived and thus no \Cref{theorem:closed_form_linear_model} analogue is needed. This is in fact true for any other additive featurization of persistence diagram such as Betti curves \cite{Umeda2017}, total persistence \cite{CohenSteiner2010}, Euler characteristic curves \cite{Richardson2014}. We illustrated how to derive this for persistence images, but similar arguments can be used to derive it for any additive games involving persistence diagrams.
	
	\begin{remark}
	The additivity of persistence images and the straightforward computation of Shapley values rely on each pixel being an unnormalized sum. Many implementations of persistence images normalize the image, e.g. dividing by the total mass so that the pixel values sum to $1$. If the persistence image of every coalition is normalized as well, a pixel’s value becomes a ratio of two additive games, 
	\[I_S(\rho_D)_p \od \dfrac{\sum_{i\in S}f(u_i)\kappa_{(i,p)}}{\sum_{i\in S}f(u_i)},\]
and removing a player now changes the denominator seen by every other point. This introduces nontrivial interactions between players. The game is no longer additive (or even monotone in general), so the trivial argument above no longer applies. Computing $\varphi_i$ exactly would require a fresh derivation of a potential closed form solution. Short of that,
the normalized case falls back to the same subsampling approach used earlier in this manuscript (\Cref{section:nonlinear_models}).
	\end{remark}
			
	\section{Discussion}
	\label{section:discussion}
	
	This manuscript introduces the first axiomatic approach, using Shapley values, towards explainable TDA. We focused on explaining models trained on persistence landscapes. We gave a closed form solution for allocating credit to persistence diagram points based on landscape games (\Cref{theorem:closed_form_landscapes}). Combining this with predictions made by linear models, we can explain how a linear model trained on persistence landscapes is making predictions (\Cref{theorem:closed_form_linear_model}). Our method works for diagrams with different numbers of points, unlike in machine learning where the dimension of features is fixed. Our credit allocation is stable (\Cref{theorem:stability_for_landscape_games,theorem:stability_nonlinear_model,corollary:landscape_game_stability}), a desirable property in TDA. For nonlinear models, we gave estimates on the number of permutations needed to achieve a desired level of approximation (\Cref{theorem:shapley_values_error_bound}). Regardless of the type of the model, the Shapley values allocated to persistence diagram points satisfy the four fairness axioms (\Cref{corollary:fairness_axioms}). All of these results are also experimentally verified (\Cref{section:experiments}). We believe this work gives a new perspective in TDA and we think it opens several new avenues of research in the field. We suggest a list of interesting research directions that would further develop the theory of explainable TDA.
	
		\begin{enumerate}[left=0pt]
			\item Can we calculate exact Shapley values of persistence diagram points for at least some classes of nonlinear models? Recall that for tree ensembles on any feature set, the Shapley values can be computed exactly in polynomial time (\Cref{section:tree_ensembles}). This would for example allow us to compute Shapley values for each landscape coordinate $(k,m)$ exactly for tree ensembles. However, the pairwise interactions of these coordinates in the nonlinear tree ensemble and how to pull back this credit allocation to the points in the persistence diagram is not clear. This raises the question, is it possible to design nonlinear models where the interaction between landscape coordinates is ``well understood".
			\item LandscapeSHAP allows us to allocate credit for a model's prediction to persistence diagram points. However, one can imagine that allocating credit to certain regions in the point clouds themselves can be of great importance as well. We did not spend any effort in this work trying to pull back credit allocation from persistence diagram points onto the point clouds themselves, for example like in \cite{bubenik2026explainabletopologicaldataanalysis}. Is it possible to define a more axiomatic way to allocate credit to regions in the point clouds themselves? Ideally, this should be ``stable", ``fair" and ``computable" (or at least have known bounds on approximations).  
			\item In this work we showed how the persistence landscape can be viewed as a collection of cooperative games, $v_{(k,m)}$, for each landscape coordinate $(k,m)$. We also briefly showed how other featurizations can be viewed as collections of cooperative games. Developing new featurizations of persistence diagrams for machine learning is an active area of research, and our work adds another dimension to this effort: if one needs to explain models trained on featurizations of persistence diagrams, one ought to design those featurizations so that they form a collection of cooperative games, since this is an axiomatically justified route to explainability. Therefore, designing featurizations in TDA should take explainability into account going forward.			
		\end{enumerate}
		
		\section*{Acknowledgement}
		The author used Claude (Anthropic) to assist with drafting and revising portions of this manuscript, checking proofs, and developing the accompanying code. All mathematical claims, proofs, and conclusions were reviewed and verified by the author, who takes full responsibility for the content.  This material is based upon work supported by the National Science Foundation under Grant No. DMS-2424556 while the author was in residence at the Institute for Computational and Experimental Research in Mathematics in Providence, RI, during the Foundations of Computational Geometry and Topology workshop program. The author in particular would also like to thank Peter Bubenik for helpful discussions about explainable TDA during this workshop.

	\printbibliography
	
\end{document}